\documentclass[12pt,reqno]{amsart}

\usepackage{amsmath,amsfonts,amsthm,amssymb,color}
\usepackage{graphicx}
\usepackage[T1]{fontenc}
\usepackage{enumerate}
\usepackage{pdfsync}

\newcommand{\be}{\mathbf{E}}
\newcommand{\bp}{\mathbf{P}}

\newcommand{\dom}{{\rm Dom}}

\newcommand{\id}{{\rm Id}}

\newcommand{\iot}{\int_{0}^{t}}

\newcommand{\ott}{[0,T]}

\newcommand{\1}{{\bf 1}}

\newcommand{\cac}{{\mathcal C}}

\newcommand{\ce}{{\mathcal E}}
\newcommand{\cf}{{\mathcal F}}

\newcommand{\ch}{{\mathcal H}}

\newcommand{\cl}{{\mathcal L}}

\newcommand{\cn}{{\mathcal N}}

\newcommand{\cs}{{\mathcal S}}

\newcommand{\al}{\alpha}
\newcommand{\ga}{\gamma}
\newcommand{\ep}{\varepsilon}
\newcommand{\ka}{\kappa}
\newcommand{\la}{\lambda}

\newcommand{\oom}{\Omega}
\newcommand{\si}{\sigma}
\newcommand{\te}{\theta}
\newcommand{\tte}{\Theta}
\newcommand{\vp}{\varphi}

\newcommand{\D}{{\mathbb D}}

\newcommand{\R}{{\mathbb R}}

\newcommand{\lla}{\left\langle}
\newcommand{\rra}{\right\rangle}

\newcommand{\lp}{\left(}
\newcommand{\rp}{\right)}
\newcommand{\lc}{\left[}
\newcommand{\rc}{\right]}
\newcommand{\lln}{\left|}
\newcommand{\rrn}{\right|}

\newcommand{\bean}{\begin{eqnarray*}}
\newcommand{\eean}{\end{eqnarray*}}
\newcommand{\ben}{\begin{enumerate}}
\newcommand{\een}{\end{enumerate}}
\newcommand{\beq}{\begin{equation}}
\newcommand{\eeq}{\end{equation}}

\newtheorem{theorem}{Theorem}[section]

\newtheorem{hypothesis}[theorem]{Hypothesis}
\newtheorem{lemma}[theorem]{Lemma}
\newtheorem{notation}[theorem]{Notation}

\newtheorem{proposition}[theorem]{Proposition}

\theoremstyle{remark}
\newtheorem{remark}[theorem]{Remark}
\newtheorem{example}[theorem]{Example}

\begin{document}

\title{On inference for fractional differential equations}
\author[A. Chronopoulou \and S. Tindel]{Alexandra Chronopoulou \and Samy Tindel}
\date{\today}

\thanks{S. Tindel is partially supported by the ANR grant ECRU. Both authors are part of the BIGS (Biology, Genetics and Statistics) team at INRIA Nancy}

\subjclass[2010]{Primary 60H35; Secondary 60H07, 60H10, 65C30, 62M09}

\date{\today}

\keywords{Fractional Brownian motion,  Stochastic differential equations, Malliavin calculus, Inference for stochastic processes}

\begin{abstract}
Based on Malliavin calculus tools and approximation results, we show how to compute a maximum likelihood type estimator for a rather general differential equation driven by a fractional Brownian motion with Hurst parameter $H>1/2$. Rates of convergence for the approximation task are provided, and numerical experiments show that our procedure leads to good results in terms of estimation.
\end{abstract}

\address{
{\it Alexandra Chronopoulou, Samy Tindel:}
{\rm Institut \'Elie Cartan Nancy, B.P. 239,
54506 Vandoeuvre-l\`es-Nancy Cedex, France}.
{\it Email: }{\tt Alexandra.Chronopoulou@iecn.u-nancy.fr, tindel@iecn.u-nancy.fr}
}

\maketitle

\section{Introduction}

In this introduction, we first try to motivate our problem and outline our results. We also argue that only a part of the question can be dealt with in a single paper. We briefly sketch a possible program for the remaining tasks in a second part of the introduction.

\subsection{Motivations and outline of the results}
\label{sec:intro-outline}
The inference problem for diffusion processes is now a fairly well understood problem. In particular, during the last two decades, several advances have allowed to tackle the problem of inference based on discretely observed diffusions \cite{DG,Pe,So}, which is of special practical interest. 

\smallskip

More specifically, consider a family of stochastic differential equations of the form
\begin{equation}\label{eq:sde}
Y_t=a+\iot \mu(Y_s;\te) \, ds +\sum_{l=1}^{d}\iot \si^l(Y_s;\te)\, dB_s^l, \qquad t\in\ott,
\end{equation}
where $a\in\R^m$, $\mu(\cdot;\te):\R^m\to\R^m$ and $\si(\cdot;\te):\R^m\to\R^{m,d}$ are smooth enough functions, $B$ is a $d$-dimensional Brownian motion and $\te$ is a parameter varying in a subset $\Theta\subset\R^q$. If one wishes to identify $\te$ from a set of discrete observations of $Y$, most of the methods which can be found in the literature are based on (or are closely linked to) the maximum likelihood principle. Indeed, if $B$ is a Brownian motion and $Y$ is observed at some equally distant instants $t_i=i\tau$ for $i=0,\ldots,n$, then the log-likelihood of a sample $(Y_{t_{1}},\ldots,Y_{t_{n}})$ can be expressed as 
\begin{equation}\label{eq:log-lik}
\ell_n(\te)=\sum_{i=1}^{n} \ln\lp p\lp  \tau,Y_{t_{i-1}},Y_{t_{i}};\, \te\rp \rp,
\end{equation}
where $p$ stands for the transition semi-group of the diffusion $Y$. If $Y$ enjoys some ergodic properties, with invariant measure $\nu_{\te_0}$ under $\bp_{\te_0}$, then we get
\begin{equation}\label{eq:cvgce-log-lik-diffusion}
{\rm a.s.-}\lim_{n\to\infty}\frac{1}{n}\ell_n(\te)
=\be_{\te_0}\lc p\lp  \tau,Z_1,Z_2;\, \te\rp\rc\triangleq J_{\te_0}(\te),
\end{equation}
where $Z_1\sim\nu_{\te_0}$ and $\cl(Z_2|\, Z_1)=p(\tau,Z_1,\cdot\, ;\,\te)$. Furthermore, it can be shown in a general context that $\te\mapsto J_{\te_0}(\te)$ admits a maximum at $\te=\te_0$. This opens the way to a MLE analysis which is similar to the one performed in the case of i.i.d observations, at least theoretically.

\smallskip

However, in many interesting cases, the transition semi-group $p$ is not amenable to explicit computations, and thus expression (\ref{eq:log-lik}) has to be approximated in some sense. The most common approach, advocated for instance in \cite{Pe}, is based on a linearization of each $p(  \tau,Y_{t_{i-1}},Y_{t_{i}};\, \te)$, which transforms it into a Gaussian density
$$
\cn\lp Y_{t_{i-1}}+\mu(Y_{t_{i-1}};\te) \,\tau, \, \si\si^*(Y_{t_{i-1}};\te)\,\tau\rp.
$$
This linearization procedure is equivalent to the approximation of equation (\ref{eq:sde}) by an Euler (first order) numerical scheme. Refinements of this procedure, based on Milstein type discretizations, are proposed in  \cite{DG}.

\smallskip

Some special situations can be treated differently (and often more efficiently): for instance, in case of a constant diffusion coefficient, the continuous time likelihood can be computed explicitly by means of Girsanov's theorem. When the dimension of the driving Brownian motion $B$ is $d=1$, one can also apply It\^o's formula in order to be back to an equation with constant diffusion coefficient, or use Doss-Sousman representation of solutions to (\ref{eq:sde}). Let us also mention that statistical inference for SDEs driven by L\'evy processes is currently  intensively investigated, with financial motivations in mind.

\smallskip

The current article is concerned with the estimation problem for equations of the form~(\ref{eq:sde}), when the driving process $B$ is a fractional Brownian motion. Let us recall that a fractional Brownian motion $B$ with Hurst parameter $H\in(0,1)$, defined on a complete probability space $(\oom,\cf,\bp)$, is a $d$-dimensional centered Gaussian process. Its law is thus characterized by its covariance function, which is given by
$$
\be \lc B_t^i B_s^j \rc= \frac 12 \lp t^{2H} + s^{2H} - |t-s|^{2H}  \rp \, \1_{(i=j)},
\qquad s,t\in\R_+.
$$
The variance of the   increments of $B$ is  then given by
$$
\be\lc  \lp  B_t^i-B_s^i \rp^2\rc = |t-s|^{2H}, \qquad s,t\in\R_+, \quad i=1,\ldots, d,
$$
 and this implies that  almost surely the fBm paths are $\gamma$-H\"{o}lder
  continuous for any $\gamma<H$. Furthermore, for $H=1/2$, fBm coincides with the usual Brownian motion,
   converting the family $\{B^H;\, H\in(0,1)\}$ into the most natural generalization of this classical process.

\smallskip

In the last decade, some important advances have allowed to solve \cite{NR,Za} and understand \cite{HN,NS} differential systems driven by fBm for $H\in(1/2,1)$. The rough paths machinery also allows to handle fBm with $H\in(1/4,1/2)$, as nicely explained in \cite{FV,Gu,lejay,LyonsBook}. However, the irregular situation $H\in(1/4,1/2)$ is not amenable to useful moments estimates for the solution $Y$ to (\ref{eq:sde}) together with its Jacobian (that is the derivative with respect to the initial condition). This is why we concentrate, in the sequel, on the simpler case $H>1/2$ for our estimation problem. In any case, many real world noisy systems are currently modeled by equations like (\ref{eq:sde}) driven by fBm, and this is particularly  present in the Biophysics literature, as assessed by \cite{KS,OTHP}, or for Finance oriented applications as in \cite{Ch,Gua,HOk,HOS,Ro,WTT}. This leads to a demand for rigorous estimation procedures for SDEs driven by fractional Brownian motion, which is the object of our paper.

\smallskip

Concerns about the inference problem for fractional diffusion processes started a decade ago with the analysis of fractional Ornstein-Uhlenbeck processes \cite{KL}. Then a more recent representative set of references on the topic includes~\cite{PL,TV}. More specifically, \cite{TV} handles the case of a one-dimensional equation of the form
\begin{equation}\label{eq:sde-cipi-fred}
Y_t=a+ \te \iot \mu(Y_s) \, ds + B_t,  \qquad t\in\ott,
\end{equation}
where $\mu$ is regular enough, and where $B$ is a fBm with $H\in(0,1)$. The simple dependence on the parameter $\te$ and the fact that an additive noise is considered enables the use of Girsanov's transform in order to get an exact expression for the MLE. Convergence of the estimator is then obtained through an extensive use of Malliavin calculus.

\smallskip

As far as \cite{PL} is concerned, it is focused on the case of a polynomial equation, for which the exact moments of the solution can be computed. The estimator relies then on a generalization of the moment method, which tries to fit empirical moments of the solution with their theoretical value. The range of application of this method is however confined to specific situations, for the following reasons:
\begin{itemize}
\item 
It assumes that $N$ independent runs of equation (\ref{eq:sde}) can be obtained, which is usually not the case.
\item
It hinges on multiple integrals computations, which are time consuming and are avoided in most numerical schemes. 
\end{itemize}
As can be seen from this brief review, parameter estimation for rough equations is still in its infancy. We shall also argue that it is a hard problem.

\smallskip

Indeed, if one wishes to transpose the MLE methods used for diffusion processes to the fBm context,  an equivalent of the log-likelihood functions  (\ref{eq:log-lik}) should first be produced. But the covariance structure of $B$ is quite complex and the attempts to put the law of $Y$ defined by (\ref{eq:sde}) into a semigroup setting are cumbersome, as illustrated by \cite{BC,HO,NNRT}. We have thus decided to consider a highly simplified version of the log-likelihood. Namely, still assuming that $Y$ is observed at  a discrete set of instants $0<t_1<\cdots<t_n<\infty$, set
\begin{equation}\label{eq:lln-rough}
\ell_n(\te)=\sum_{i=1}^{n} \ln\lp f(t_i,Y_{t_i};\te)\rp,
\end{equation}
where we suppose that under $\bp_\te$ the random variable $Y_{t_i}$ admits a density $z\mapsto f(t_i, z;\te)$. Notice that in case of an elliptic diffusion coefficient $\si$ the density $f(t_i,\cdot;\te)$ is strictly positive, and thus expression (\ref{eq:lln-rough}) makes sense by a straightforward application of \cite[Proposition 19.6]{FV}. However, the successful replication of the strategy implemented for Brownian diffusions (that we have tried to summarize above) relies on some highly non trivial questions: existence of an invariant measure for equation (\ref{eq:sde}), rate of convergence to this invariant measure, convergence of expressions like (\ref{eq:lln-rough}), characterization of the limit in terms of $\te$ as in (\ref{eq:cvgce-log-lik-diffusion}), to mention just a few. We shall come back to these considerations in the next section, but let us insist at this point on the fact that all those questions would fit into a research program over several years. 

\smallskip

Our aim in this paper is in a sense simpler: we assume that quantities like (\ref{eq:lln-rough}) are meaningful for estimation purposes. Then we shall implement a method which enables to compute $\ell_n(\te)$ and optimize it in $\te$, producing thus a pseudo MLE estimator. We focused first on this specific aspect of the problem for the following reasons:
\begin{enumerate}
\item 
From a statistical point of view, it is obviously essential to obtain a computationally efficient estimation procedure. This will allow us for instance to evaluate numerically the accuracy of our method. 
\item
The procedure itself is nontrivial, and requires the use of advanced stochastic analysis tools: probabilistic representation of the density, Malliavin type integration by parts, Stratonovich-Skorohod correction terms, discretization of systems of pathwise stochastic differential equations for instance.
\end{enumerate}
We have thus decided to tackle the implementation issues first. If it turns out to be satisfying, we shall then try to proceed to a full justification of our method.

\smallskip

Let us also mention that it might not be clear to the reader that $\ell_n(\te)$ can be meaningful in terms of statistical estimation, since  it only involves evaluations at single points $Y_{t_i}$. However our numerical experiments indicate that this quantity behaves nicely for our purposes. Moreover, it will become clear from the forthcoming computations that our methodology can be extended to handle quantities like 
$$
\tilde\ell_n(\te):=\sum_{i=1}^{n} \ln\lp f(t_i,t_{i+1},Y_{t_i},Y_{t_{i+1}};\te)\rp,
$$
where $f(s,t,x,z;\te)$ stands for the density of the couple $(Y_s,Y_t)$. This kind of pseudo log-likelihood is obviously closer in spirit to the diffusion case. Densities of tuples could also be considered at the price of technical complications.

\smallskip

Let us now try to give a flavor of the kind of result we shall obtain in this article, in a very loose form:
\begin{theorem}
Consider Equation (\ref{eq:sde}) driven by a $d$-dimensional fractional Brownian motion $B$ with Hurst parameter $H>1/2$. Assume $\mu$ and $\si$ are smooth enough coefficients, and that $\si \si^*$ is strictly elliptic. For a sequence of times $t_0<\cdots <t_n<\infty$, let  $y_{t_{i}}$, $i=1,\ldots,n$ be the corresponding observations. Then:

\smallskip

\noindent
\emph{(i)} The gradient of the log-likelihood function admits the following probabilistic representation:
$\nabla_{l}\ell_n(\te) =\sum_{i=1}^{n}\frac{V_i(\te)}{W_i(\te)}$, with
\begin{equation}\label{eq:expr-W-i-intro}
W_i(\te)=\mathbf{E} \biggl[ \mathbf{1}_{(Y_{t_{i}}(\theta)>y_{t_{i}})}\; H_{(1,\ldots,m)}\Bigl( Y_{t_{i}}(\theta) \Bigr) \biggr]
\end{equation}
where  $H_{(j_1,\ldots,j_n)}( Y_{t_{i}}(\theta))$ is an expression involving Malliavin derivatives an Skorohod integrals of $Y(\theta)$. A similar expression is also available for $V_i(\te)$.

\smallskip

\noindent
\emph{(ii)} A computational procedure is constructed in order to obtain $H_{(1,\ldots,m)}( Y_{t_{i}}(\theta))$ in a suitable way.

\smallskip

\noindent
\emph{(iii)} When $Y_t$ is replaced by its Euler scheme approximation with step $T/M$ and expected values in (\ref{eq:expr-W-i-intro}) are approximated thanks to $N$ Monte Carlo steps, we show that 
\begin{itemize}
\item 
$N$ can be chosen in function of $M$ in an optimal way (see Proposition \ref{prop:MCvsEU}).
\item
The corresponding approximation of $\nabla_{l}\ell_n(\te)$ converges to the real one with rate $n^{-(2\ga-1)}$ for any $1/2<\ga<H$.
\end{itemize}
\end{theorem}

All those results are stated in a more rigorous way in the remainder of the article. 

\smallskip

Here is how our article is structured: we give some preliminaries and notations on Young and Malliavin calculus for fractional Brownian motion at Section \ref{sec:preliminaries}. The probabilistic representation for the log-likelihood is given at Section \ref{sec:exp-log-likelihood}. Discretization procedures are designed at Section \ref{sec:discret-likelihood}, and finally numerical examples are given at Section \ref{sec:num-examples}.

\subsection{Remaining open problems}
We emphasized above the fact that only a part of the problem at stake was going to be solved in the current article. We now briefly sketch  the remaining tasks to be treated.

\smallskip

The most important obstacle in order to fully justify our methodology is to get a suitable convergence theorem for $\ell_n(\te)/n$, where $\ell_n(\te)$ is defined by (\ref{eq:lln-rough}). In a natural way, this should be based on some strong ergodicity properties for $Y_t$. After a glance at the literature on ergodicity for fractional systems, one can distinguish two cases:

\smallskip

\noindent
\textit{(i)} When $\si(\cdot;\te)$ is constant, the convergence of $\cl(Y_t)$ as $t\to\infty$ is established in \cite{Ha}, with a (presumably non optimal) rate of convergence $t^{-1/8}$.

\smallskip

\noindent
\textit{(ii)} For a general smooth and elliptic coefficient $\si$, only the uniqueness of the invariant measure is shown in \cite{HO}, with an interesting extension to the hypoelliptic case in \cite{HP}. Nothing is known about the convergence of $\cl(Y_t)$, not to mention rates.

\smallskip

\noindent
This brief review already indicates that the convergence to invariant measures is still quite mysterious for fractional differential equations, at least for a non constant coefficient $\si$. Moreover, recall that if $\nu(\te)$ stands for the invariant measure corresponding to the system with coefficients $\mu(\cdot;\te),\si(\cdot;\te)$, we also wish to retrieve some information on the dependence $\te\mapsto\nu(\te)$ (See \cite{HM} for some partial results in this direction).

\smallskip

Let us mention another concrete problem: even in the case of a constant $\si$, the convergence of $\cl(Y_t)$ to an invariant measure $\nu(\te)$ is proven in \cite{Ha} in the total variation sense. In terms of the density $p(t,x;\te)$ of $Y_t$, it means that $p(t,\cdot;\te)$ converges to the density of $\nu$ in $L^1$ topology. However, in order to get a limit for $\ell_n(\te)/n$, one expects to use at least a convergence in some Sobolev space $W^{\al,p}$ for $\al,p$ large enough. 

\smallskip

One possibility in order to get this sharper convergence is to bound first the density $p(t,\cdot;\te)$ in another Sobolev space $W^{\al',p'}$ and then to use interpolation theory.  It seems thus sufficient to obtain Gaussian bounds on $p(t,\cdot;\te)$, uniformly in $t$. In case of Brownian diffusions, these Gaussian bounds are obtained by analytic tools, thanks to the Markov property. This method being obviously not available for systems driven by fBm, a possible inspiration is contained in the upper Gaussian bounds for the stochastic wave equation which can be found in \cite{DN}. The latter technical results stem from an intensive use of Malliavin calculus, which should also be invoked in our case, and notice the recent efforts~\cite{BO,BOT} in this direction.

\smallskip

Finally, let us mention that it seems possible to produce some reasonable convergent parametric estimators  for equations driven by fBm in a rather general context. Among the methods which can be adapted from the diffusion case with the current stochastic analysis techniques, let us mention the least square estimator of \cite{Ka}, as well as the local asymptotic normality property shown in \cite{Go}.
However, it seems obvious that the road to a complete picture of parameter estimation for stochastic equations driven by fBm is still hard and long. We hope to complete it in some subsequent communications.

\section{Preliminaries and notations}
\label{sec:preliminaries}

As mentioned in the introduction, we are concerned with equations driven by a $d$-dimensional fractional Brownian motion $B$. We recall here some basic facts about the way to solve those equations, and some Malliavin calculus tools which will be needed later on. Let us introduce first some general notation for H\"older type spaces:
\begin{notation}
We will denote by $\cac^{\al}(V)$ the set of $V$-valued $\al$-H\"older functions for any $\al\in(0,1)$, and by $\cac_{b}^{n}(U;V)$ the set of $n$ times differentiable functions, bounded together with all their derivatives, from $U$ to $V$. In the previous notation, $U$ and $V$ stand for two finite dimensional vector spaces. The state space $V$ can be omitted for notational sake when its value is non ambiguous. When we want to stress the fact that we are working on a finite interval $\ott$, we write $\cac_{T}^{\al}(V)$ for the space of $\al$-H\"older functions $f$ from $\ott$ to $V$. The corresponding H\"older norms shall be denoted by $\|f\|_{\al,T}$.
\end{notation}

\subsection{Differential equations driven by fBm}

Recall that the equation we are interested in is of the form (\ref{eq:sde}).  Before stating the assumptions on our coefficients we need an additional notation:
\begin{notation}\label{not:partialDER}
For $n,p\ge 1$, a function $f\in\cac^{p}(\R^{n};\R)$ and any tuple $(i_1,\ldots i_p)\in\{1,\ldots,d\}^{p}$, we set $\partial_{i_1\ldots i_p} f$ for $\frac{\partial^{p} f}{\partial x_{i_1}\ldots \partial x_{i_p}}$. Similarly, consider a function $g_{\theta}\in\cac^{p}(\tte;\R)$, for $n,p\ge 1$ and a vector of parameters $\theta \in \Theta \subset \R^{q} $. For any tuple $(i_1,\ldots i_p)\in\{1,\ldots,q\}^{p}$, we set $\nabla_{i_1\ldots i_p} g_{\theta}^{i}$ for $\frac{\partial^{p} g_{\theta}^{i}}{\partial \theta_{i_1}\ldots \partial \theta_{i_p}}$, where $i=1,\ldots, n$. 
\end{notation}
Using this notation, we work under the following set of assumptions:
\begin{hypothesis}\label{hyp:coeff-sde}
For any $\te\in\tte$, we assume that $\mu(\cdot;\te):\R^m\to\R^m$ and $\si(\cdot;\te):\R^m\to\R^{m,d}$ are $\cac_b^{2}$ coefficients. Furthermore, we have
\begin{equation*}
\sup_{\te\in\tte} \sum_{l=0}^{2} \sum_{1\le i_1,\ldots,i_l \le q} 
\|\nabla_{i_1 \cdots i_l}^{l}\mu(\cdot;\te)\|_{\infty} + \|\nabla_{i_1 \cdots i_l}^{l}\si(\cdot;\te)\|_{\infty}
 <\infty.
\end{equation*}
\end{hypothesis}

\smallskip

When equation (\ref{eq:sde}) is driven by a fBm with Hurst parameter $H>1/2$ it can be solved, thanks to a fixed point argument, with the stochastic integral interpreted in the  (pathwise) Young sense (see e.g. \cite{Gu}). Let us recall that Young's integral can be defined in the following way:
\begin{proposition}\label{prop:young-intg}
Let $f\in \cac^\ga$, $g\in \cac^\kappa$ with $\ga+\kappa>1$, and $0\le s\le t\le 1$. Then the  integral $\int_s^t g_\xi\; df_\xi$ is well-defined as limit of Riemann sums along partitions of $[s,t]$. Moreover, the following estimation is fulfilled:
\beq
\left|\int_s^t g_\xi\; df_\xi\right| \leq C \|f\|_\ga \|g\|_\kappa |t-s|^\ga,
\label{y}
\eeq
where the constant $C$ only depends on $\ga$ and $\kappa$. A sharper estimate is also available:
\begin{equation}\label{eq:ineq-young-sharp}
\left|\int_s^t g_\xi\; df_\xi\right| \leq |g_s| \, \|f\|_\ga |t-s|^\ga
+ c_{\ga,\ka}  \|f\|_\ga \|g\|_\kappa |t-s|^{\ga+\ka}.
\end{equation}
\end{proposition}

\smallskip

With this definition in mind and under assumptions \ref{hyp:coeff-sde}, we can solve our differential system of interest, and the following moments bounds are proven in \cite{FV,HN}:
\begin{proposition}\label{prop:moments-sdes}
Consider a fBm $B$ with Hurst parameter $H>1/2$. Then:

\smallskip

\noindent
(1) 
Under Hypothesis \ref{hyp:coeff-sde}, equation (\ref{eq:sde}) driven by $B$ admits a unique $\beta$-H\"older continuous solution $Y$, for any $\beta<H$.

\smallskip

\noindent
(2) 
Furthermore,
\begin{equation*}
\|Y\|_{T,\beta} \le  |a|+ c_{f,T} \|B\|_{\beta,T}^{1/\beta}.
\end{equation*}

\smallskip

\noindent
(3) If we denote by $Y^a$ the solution to (\ref{eq:sde}) with initial condition $a$, then
\begin{equation*}
\|Y^b-Y^a\|_{T,\beta} \le |b-a| \, \exp\lp c_{f,T} \|B\|_{\beta,T}^{1/\beta}\rp.
\end{equation*}

\smallskip

\noindent
(4) If we only assume that $f$ has linear growth, with $\nabla f,\nabla^2 f$ bounded, the following estimate holds true:
\begin{equation*}
sup_{t\in\ott} |Y_t| \le \lp 1+ |a|\rp \,  \exp\lp c_{f,T} \|B\|_{\beta,T}^{1/\beta}\rp.
\end{equation*}
\end{proposition}

\begin{remark}
The framework of fractional integrals is used in \cite{HN} in order to define integrals with respect to $B$. It is however easily seen to be equivalent to the Young setting we have chosen to work with.
\end{remark}

Some differential calculus rules for processes controlled by fBm will also be useful in the sequel:
\begin{proposition}\label{prop:diff-calc-rule}
Let $B$ be a $d$-dimensional fBm with Hurst parameter $H>1/2$. Consider $a,\hat a\in\R$, $b,\hat b\in\cac^{\al}_T(\R^d)$ with $\al+H>1$, and $c,\hat c\in\cac_T(\R)$ (all these assumptions are understood in the almost sure sense). Define two processes $z,\hat z$ on $\ott$ by 
\begin{equation*}
z_t=a+\sum_{j=1}^{d}\int_0^t b_u^{j}\, dB_u^{j} + \int_0^t c_u \, du, 
\quad\mbox{and}\quad
\hat z_t= \hat a+\sum_{j=1}^{d}\int_0^t \hat b_u^{j}\, dB_u^{j} + \int_0^t \hat c_u \, du.
\end{equation*}
Then for $t\in\ott$, one can decompose the product $z_t\hat z_t$ into
\begin{equation*}
z_t \, \hat z_t= a  \, \hat a + \sum_{j=1}^{n} \int_0^t \lc \hat z_u  \, b_u^{j}+ z_u  \, \hat b_u^{j}\rc \, dB_u^{j}
+ \int_0^t \lc z_u \, \hat c_u + \hat z_u c_u \rc\, du,
\end{equation*}
where all the integrals with respect to $B$ are understood in the Young sense.
\end{proposition}
The proof of this elementary and classical result is omitted here. See \cite[Proposition~2.8]{LT} for the proof of a similar rule.

\subsection{Malliavin calculus techniques}
Our representation of the density for the solution to (\ref{eq:sde}) obviously relies on Malliavin calculus tools that we proceed now to recall. As already mentioned in the introduction, on a finite interval $\ott$ and for some fixed $H\in(1/2,1)$, we consider $(\oom,\cf,P)$ the canonical probability space associated with a fractional
Brownian motion with Hurst parameter $H$. That is,  $\oom=\cac_0(\ott;\R^d)$ is the Banach space of continuous functions
vanishing at $0$ equipped with the supremum norm, $\cf$ is the Borel sigma-algebra and $P$ is the unique probability
measure on $\oom$ such that the canonical process $B=\{B_t, \; t\in [0,T]\}$ is a $d$-dimensional fractional Brownian motion with Hurst
parameter $H$. Remind that this means that $B$ has $d$ independent coordinates, each one being a centered Gaussian process with covariance
$
R_H(t,s)=\frac 12 (s^{2H}+t^{2H}-|t-s|^{2H}).
$

\subsubsection{Functional spaces}
Let $\ce$ be the space of $d$-dimensional elementary functions on $\ott$:
\begin{multline}\label{eq:def-elem-fct}
\ce=\Big\{ f=(f_1,\ldots,f_d);\,\,f_j=\sum_{i=0}^{n_j-1} a_i^j
\1_{[t_i^j, t_{i+1}^j)}\,, \quad
0=t_0<t_1^j<\cdots<t_{n_j-1}^j<t_{n_j}^j=T,\\
\text{ for }j=1,\ldots,d\Big\}\,.
\end{multline}
We call $\ch$ the completion of $\ce$ with respect to the semi-inner product
\[
\lla f,\, g\rra_{\ch}=\sum_{i=1}^{d} \lla f_{i},\, g_{i}\rra_{\ch_{0}},
\quad\mbox{where}\quad
\langle \1_{[0,t]}, \1_{[0,s]} \rangle_{\ch_{0}} := R(s,t), \quad s,t \in [0,T].
\]
Then, one constructs an isometry $K^*_H: \ch \rightarrow  L^2([0,1];\R^d)$  such that 
$$
K^*_H\lp \1_{[0,t_{1}]},\ldots,\1_{[0,t_{d}]}\rp 
= \lp \1_{[0,t_1]} K_H(t_1,\cdot),\ldots, \1_{[0,t_d]} K_H(t_d,\cdot)\rp,
$$ 
where the kernel $K_H$ is given by 
\[
K_H(t,s)= c_H s^{\frac 12 -H} \int_s^t (u-s)^{H-\frac 32} u^{H-\frac 12} \, du
\]
and verifies that $R_H(t,s)= \int_0^{s\land t} K_H(t,r) K_H(s,r)\, dr$, for some constant $c_H$. Moreover, let us observe that $K^*_H$ can be represented 
in the following form: for $\vp=(\vp_1,\ldots,\vp_d)\in\ch$, we have $K^*_H \vp$
\[
K^*_H \vp=\lp  K^*_H \vp^1,\ldots,K^*_H\vp^d \rp,
\quad\mbox{where}\quad
[K^*_H \vp^i]_t = \int_t^1 \vp_r^i \partial_r K_H(r,t) \, dr.
\]

\subsubsection{Malliavin derivatives}

Let us start by defining the Wiener integral with respect to $B$: for any element $f$ in $\ce$ whose expression is given as in (\ref{eq:def-elem-fct}), we define the Wiener integral of $f$ with respect to $B$ as
\[
B(f):=\sum_{j=1}^d\sum_{i=0}^{n_j-1} a_i^j (B_{t_{i+1}^j}^{j}
-B_{t_i^j}^{j})\,.
\]
We also denote this integral as $ \int_0^T f_{t} dB_t$, since it coincides with a pathwise integral with respect to $B$.

\smallskip

For $\theta:\R\rightarrow \R$, and $j\in\{1,\ldots,d\}$, denote by
$\theta^{[j]}$ the function with values in $\R^d$ having all the
coordinates equal to zero except the $j$-th coordinate that equals
to $\theta$. It is readily seen that
$$
\be\lc B\lp \1_{[0,s)}^{[j]}\rp \, B\lp \1_{[0,t)}^{[k]}\rp \rc
=\delta_{j,k}R_{s,t}.
$$
This definition can be extended by linearity and closure to elements of $\ch$, and we obtain the relation
$$
\be\lc B(f) \, B(g)\rc =\langle f,g\rangle_{\ch},
$$
valid for any couple of elements $f,g\in\ch$. In particular, $B(\cdot)$ defines an isometric map from $\ch$  into a subspace of $L^2(\Omega)$.

\smallskip

We can now proceed to the definition of Malliavin derivatives. With this notation \ref{not:partialDER} in hand, let us consider $\cs$ be the 
family of smooth functionals $F$ of the form
\begin{equation}\label{eq:def-smooth-fct}
F=f(B(h_1),\dots,B(h_n)),
\end{equation}
where $h_1,\dots,h_n\in \ch$, $n\geq 1$, and $f$ is a smooth function with polynomial growth, together with all its derivatives. Then, the Malliavin derivative of such a functional $F$ is the $\ch$-valued random variable defined by
\[
D F= \sum_{i=1}^n \partial_{i} f(B(h_1),\dots,B(h_n)) \, h_i.
\]
For all $p>1$, it is known that the operator $D$ is closable from $L^p(\oom)$ into $L^p(\oom; \ch)$ (see e.g. \cite[Section 1]{N-bk}).
We will still denote by $D$ the closure of this operator, whose domain is usually denoted by $\D^{1,p}$ and is defined 
as the completion of $\cs$ with respect to the norm
\[
\|F\|_{1,p}:= \left( E(|F|^p) + E( \|D F\|_\ch^p ) \right)^{\frac 1p}.
\]
It should also be noticed that partial Malliavin derivatives with respect to each component $B^{j}$ of $B$ will be invoked: they are defined, for a functional $F$ of the form (\ref{eq:def-smooth-fct}) and $j=1,\dots,d$, as
\begin{equation*}
D^j F=\sum_{i=1}^n  \partial_{i} f(B(h_1),\dots,B(h_n)) h_i^{[j]},
\end{equation*}
and then extended by closure arguments again. We refer to \cite[Section 1]{N-bk} for the definition of higher derivatives and Sobolev spaces $\D^{k,p}$ for $k>1$. Another essential object related to those derivatives is the so-called Malliavin matrix of a $\R^m$-valued random variable $F\in\D^{1,2}$, defined by
\begin{equation}\label{eq:def-mall-matrix}
\ga_{F}=
\biggl (\Bigl \langle DF^{i}, DF^{j} \Bigr \rangle \biggr)_{1\leq i,j \leq m}.
\end{equation}

\subsubsection{Skorohod integrals}
\label{sec:sko-integrals}

We will denote by $\delta$ the adjoint of the operator $D$ (also referred to as the {\it divergence operator}). This operator  is closed and its domain, denoted by
$\dom(\delta)$, is the set of $\ch$-valued
square integrable random variables $u\in L^2(\Omega;\ch)$ such that
$$
|\be\lc\langle D F,u\rangle_{\ch}\rc |\le C\,\|F\|_2,
$$
for all $F\in\D^{1,2}$, where $C$ is some constant depending on
$u$. Moreover, for $u\in\dom(\delta)$,  $\delta(u)$ is the element of $L^2(\Omega)$ characterized by the duality relationship:
\begin{equation}\label{duality}
\be\lc F\delta(u)\rc=\be\lc\langle D F,u\rangle_{\ch}\rc,
\quad\mbox{for any}\quad F\in\D^{1,2}.
\end{equation}
The quantity $\delta(u)$ is usually called \emph{Skorohod integral} of the process $u$.

\smallskip

Skorohod integrals are obviously analytic objects, not suitable for easy numerical implementations. However, they can be related to the Young type integrals introduced at Proposition \ref{prop:young-intg}. For this, we need to define another functional space as follows:
\begin{notation}
We call $|\ch|$ the space of measurable functions $\vp:[0,T]\rightarrow \R^d$ such that
\[
\|\vp\|^2_{|\ch|}:= c_H \int_0^1 \int_0^1 |\vp_r| |\vp_u| |r-u|^{2H-2} dr du <+\infty,
\]
where $c_H=H(2H-1)$, and we denote by $\langle \cdot,\cdot\rangle_{|\ch|}$ the associated inner product. We also write $\D^{k,p}(|\ch|)$ for the space of $\D^{k,p}$ functionals with values in $|\ch|$.
\end{notation}
The following proposition is then a slight extension of \cite[Proposition 5.2.3]{N-bk}:
\begin{proposition}\label{prop:dom-strato}
Let $\{u_t^{ij},\; t\in [0,1]\}$, for $i=1,\ldots,m$ and $j=1,\ldots,d$,  be a stochastic process in $\D^{1,2}(|\ch|)$ such that
\beq
\sum_{j=1}^{d}\int_0^1 \int_0^1 |D_s^{j} u_t^{ij}| \, |t-s|^{2H-2} ds dt <+\infty \quad a.s.
\label{eq:an}
\eeq
We also assume that almost surely, $u$ has $\beta$-H\"older paths with $\beta+H>1$.
Then the Young integral $\sum_{j=1}^{d}\int_0^T  u_t^{ij} \, dB_t^{j}$ exists and for all $i=1,\ldots,m$ can be written as
\[
\sum_{j=1}^{d}\int_0^T u_t^{ij} \, dB_t^{j} = \delta(u^{i}) + \sum_{j=1}^{d} \int_0^T \int_0^T D_s^{j} u_t^{ij} |t-s|^{2H-2} ds dt,
\]
where $\delta(u)$ stands for the Skorohod integral of $u$.
\end{proposition}

\section{Probabilistic expression for the log-likelihood}
\label{sec:exp-log-likelihood}

Recall that we are focusing on equation (\ref{eq:sde}) driven by a $d$-dimensional fBm $B$, and that we have chosen to use expression (\ref{eq:lln-rough}) as a substitute to the log-likelihood function. We have thus reduced the initial maximization problem to the solution of $\nabla_{l}\ell_{n} (\theta) = 0 $. This will be performed numerically by means of a root approximation algorithm. 

\smallskip

Observe first that in order to define (\ref{eq:lln-rough}), the density of $Y_t(\te)$ must exist for any $t>0$. Let us thus recall the classical setting (given in \cite{HN}) under which $Y_t$ admits a smooth density:
\begin{hypothesis}\label{hyp:coeff-sde2}
Let $\mu$ and $\si$ be coefficients satisfying Hypothesis \ref{hyp:coeff-sde}. For $\xi\in\R^m$ and $\te\in\tte$, set $\al(\xi)=\si(\xi,\te) \si^*(\xi,\te)$. Then we assume that

\smallskip

\noindent
\emph{(i)} For any $k\ge 0$ and $j_1,\ldots,j_k\in\{1,\ldots,m\}$ we have
\begin{equation*}
\sup_{\te\in\tte} \sum_{l=0}^{2} \sum_{1\le p_1,\ldots,p_l \le q} 
\|\nabla_{p_1 \cdots p_l}^{l} \partial_{j_1,\ldots,j_k}^{k} \mu(\cdot;\te)\|_{\infty} 
+ \|\nabla_{p_1 \cdots p_l}^{l} \partial_{j_1,\ldots,j_k}^{k} \si(\cdot;\te)\|_{\infty}
 \le c_{k},
\end{equation*}
for a strictly positive constant $c_{k}$.

\smallskip

\noindent
\emph{(ii)} 
There exists a strictly positive constant $\ep$ such that $\langle\al(\xi;\,\te)\eta,\, \eta\rangle_{\R^m}\ge \ep |\eta|^2_{\R^m}$ for any couple of vectors $\eta,\xi\in\R^m$, uniformly in $\te\in\tte$. 
\end{hypothesis}
Then the density result for $Y_t$ can be read as follows:
\begin{theorem}
Consider the stochastic differential equation (\ref{eq:sde}) with initial condition $a\in\R^m$. Assume Hypothesis \ref{hyp:coeff-sde2} is satisfied. Then, for any $t>0$ and $\te\in\tte$, the law of $Y_t(\te)$ admits a $\cac^\infty$ density, denoted by $f(t,\cdot;\,\te)$, with respect to Lebesgue's measure.
\end{theorem}

\smallskip

In the sequel, we shall suppose that the density $f(t,\cdot;\,\te)$ exists without further mention, the aim of this section being to produce a probabilistic representation of $f(t,\cdot;\,\te)$ for computational purposes. To this aim, we shall first give the equations governing the Malliavin derivatives of the processes $Y(\te)$ and $\nabla Y(\te)$, and then use a stochastic analysis formula in order to represent our log-likelihood. We separate these tasks in two different subsections.

\subsection{Some Malliavin derivatives}
This section is devoted to a series of preliminary lemmas which will enable to formulate our probabilistic representation of $f(t,\cdot;\,\te)$. Let us first introduce a notation which will prevail until the end of the paper:
\begin{notation}
For a set of indices or coordinates $(k_1,\ldots,k_r)$ of length $r\ge 1$ and $1\le j\le r$, we denote by $(k_1,\ldots,\check{k}_j,\ldots,k_r)$ the set of indices or coordinates of length $r-1$ where $k_j$ has been omitted.
\end{notation}
We now give a general expression for the higher order derivatives of $Y_t$, borrowed from~\cite{NS}.
\begin{lemma}\label{lem:higher-mall-deriv-Y}
Assume Hypothesis \ref{hyp:coeff-sde} and \ref{hyp:coeff-sde2} hold true. For $n\ge 1$ and $(i_{1},\ldots,i_{n})\in\{1,\ldots,d\}^{n}$, denote by $D^{i_{1},\ldots,i_{n}}Y_{t}^{i}(\theta)$ the $n\textsuperscript{th}$ Malliavin derivative of $Y_{t}^{i}(\theta)$ with respect to the coordinates $B^{i_1},\ldots,B^{i_n}$ of $B$. Then $D^{i_{1},\ldots,i_{n}}Y_{t}^{i}(\theta)$, considered as an element of $\ch^{\otimes n}$, satisfies  the following linear equation: for $t \geq r_{1}\vee \cdots \vee r_{n}$,
\begin{multline}\label{eq:higher-deriv-Y-t}
D_{r_{1},\ldots,r_{n}}^{i_{1},\ldots,i_{n}} Y_{t}^{i}(\theta) = 
\sum_{p=1}^{n} \alpha^{i}_{i_p,i_{1}\ldots,\check{\imath}_p,\ldots, i_{n}}(r_{p};r_{1},\ldots,\check{r}_{p},\ldots,r_{n};\theta)\\
 + \int_{r_{1}\vee \cdots \vee r_{n}}^{t} \beta^{i}_{i_{1},\ldots,i_{n}} (s;r_{1},\ldots,r_{n};\theta) \;ds
+ \sum_{l=1}^{d} \int_{r_{1}\vee \cdots \vee r_{n}}^{t} \alpha^{i}_{l,i_{1},\ldots,i_{n}} (s;r_{1},\ldots,r_{n};\theta) \;dB_{s}^{l},
\end{multline}
where
\begin{eqnarray*}
\alpha^{i}_{j,i_{1},\ldots,i_{n}} (s;r_{1},\ldots,r_{n};\theta) &=&
\sum \sum_{k_{1},\ldots,k_{\nu}=1 }^{m} 
\partial_{k_{1}\ldots k_{\nu}}^{\nu}  \sigma^{ij}(Y_{s}(\theta);\theta) \; 
D_{r(I_{1})}^{i(I_{1})} Y_{s}^{k_{1}}(\theta) \ldots D_{r(I_{\nu})}^{i(I_{\nu})} Y_{s}^{k_{\nu}}(\theta)\\
\beta^{i}_{i_{1},\ldots,i_{n}} (s;r_{1},\ldots,r_{n};\theta) &=& 
\sum \sum_{k_{1},\ldots,k_{\nu}=1 }^{m} 
\partial_{k_{1}\ldots k_{\nu}}^{\nu}  \mu^{i} (Y_{s}(\theta);\theta )\; 
D_{r(I_{1})}^{i(I_{1})} Y_{s}^{k_{1}}(\theta) \ldots D_{r(I_{\nu})}^{i(I_{\nu})} Y_{s}^{k_{\nu}}(\theta).
\end{eqnarray*}
In the expressions above, the first sums are extended to the set of all partitions $I_{1}, \ldots, I_{\nu}$ of $\{1,\ldots,n\}$ and for any subset $K=\{i_{1},\ldots,i_{\eta}\}$ of $\{1,\ldots,n\}$ we set $D_{r(K)}^{i(K)}$ for the derivative operator $D_{r_{1},\ldots,r_{\eta}}^{i_{1},\ldots,i_{\eta}} $. Notice that $D_{r_{1},\ldots,r_{n}}^{i_{1},\ldots,i_{n}} Y_{t}^{i}(\theta) =0$ whenever $t < r_{1}\vee \cdots \vee r_{n}$.
\end{lemma}

The formulas above might seem intricate. The following example illustrate their use in a simple enough situation: 
\begin{example}
The second order derivative $D_{r_{1},r_{2}}^{1,3}Y_{t}^{2}(\theta)$ can be computed as:
\begin{eqnarray*}
D_{r_{1},r_{2}}^{1,3}Y_{t}^{2}(\theta) &=& \alpha_{1,3}^{2}(r_{1},r_{2};\theta) + \alpha_{3,1}^{2}(r_{2},r_{1};\theta) \\
&&+ \int_{r_{1}\vee r_{2}}^{t} \beta^{2}_{1,3}(s,r_{1},r_{2};\theta)\;ds + \sum_{l=1}^{d} \int_{r_{1}\vee r_{2}}^{t} \alpha_{l,1,3}^{2}(s,r_{1},r_{2};\theta) dB_{s}^{l},
\end{eqnarray*}
where
\begin{eqnarray*}
\alpha_{1,3}^{2}(r_{1},r_{2};\theta) &=& \sum_{k=1}^{m} \partial_{k} \sigma^{21}(Y_{r_{2}}(\theta);\theta)\;D_{r_{2}}^{3} Y_{r_{1}}^{k}(\theta),\\
\alpha_{3,1}^{2}(r_{2},r_{1};\theta) &=& \sum_{k=1}^{m} \partial_{k} \sigma^{23}(Y_{r_{1}}(\theta);\theta)\;D_{r_{1}}^{1} Y_{r_{2}}^{k}(\theta) 
\end{eqnarray*}
and
\begin{eqnarray*}
\beta^{2}_{1,3}(s,r_{1},r_{2};\theta) &=&  \partial^{2}_{kk} \mu^{2}(Y_{s}(\theta);\theta) D_{r_{1},r_{2}}^{1,3} Y_{s}^{k}(\theta)
+\partial_{k_{1}k_{2}}^{2} \mu^{2}(Y_{s}(\theta);\theta) D_{r_{1}}^{1} Y_{s}^{k_{1}}(\theta) D_{r_{2}}^{3} Y_{s}^{k_{2}}(\theta), \\
\alpha_{l,1,3}^{2}(s,r_{1},r_{2};\theta) &=&   \partial^{2}_{kk} \sigma^{2l}(Y_{s}(\theta);\theta) D_{r_{1},r_{2}}^{1,3} Y_{s}^{k}(\theta)
+ \partial_{k_{1}k_{2}}^{2} \sigma^{2l}(Y_{s}(\theta);\theta) D_{r_{1}}^{1} Y_{s}^{k_{1}}(\theta) D_{r_{2}}^{3} Y_{s}^{k_{2}}(\theta),
\end{eqnarray*}
where we have used the convention of summation over repeated indices.
\end{example}

Our formula for the log-likelihood will also involve some derivatives of the process $Y(\theta)$ with respect to the parameter $\theta$. The existence of this derivative is assessed below:
\begin{proposition}\label{prop:dif-theta-Y}
Under the same hypothesis as for Lemma \ref{lem:higher-mall-deriv-Y}, the random variable $Y_{t}^{i}(\te)$ is a smooth function of $\te$ for any $t\ge 0$. We denote by $\nabla_{l}Y_{t}^{i}(\te)$ the derivative of $Y_{t}^{i}(\te)$ with respect to the $l^{th}$ element of the vector of parameters $\te$. This process satisfies the following SDE:
\begin{eqnarray*}
\nabla_{l}Y_{t}^{i}(\theta) &=& 
\int_{0}^{t} [\partial_{i} \mu^{i}(Y_{u}(\theta);\theta) \nabla_{l}Y_{u}^{i}(\theta) + \nabla_{l}\mu^{i}(Y_{u}(\theta);\theta)]du \\
&& + \sum_{j=1}^{d}\int_{0}^{t}[\partial \sigma^{ij}(Y_{u}(\theta);\theta) \nabla_{l}Y_{u}^{i}(\theta) + \nabla_{l}\sigma^{ij}(Y_{u}(\theta);\theta)]dB_{u}^{j}.
\end{eqnarray*}
\end{proposition}

\begin{proof}
The proof goes exactly along the same lines as for \cite[Proposition 4]{NS}, and the details are left to the reader.

\end{proof}

 We shall also need some equations governing the Malliavin derivatives of $\nabla_{l}Y(\te)$.  This is the aim of the following lemma:
\begin{lemma}\label{lem:higher-mall-deriv-nabla-Y}
For any $l\in\{1,\ldots,q\}$ and $n\ge 1$, the process $\nabla_{l}D^{i_{1},\ldots,i_{n}}Y(\te)$ is $n$-times differentiable in the Malliavin calculus sense. Moreover, taking up the notations of Lem\-ma~\ref{lem:higher-mall-deriv-Y}, the process $\nabla_{l}D^{i_{1},\ldots,i_{n}}Y_{t}^{i}(\theta)$ satisfies the following linear equation: for $t \geq r_{1}\vee \cdots \vee r_{n}$,
\begin{multline*}
\nabla_{l}D_{r_{1},\ldots,r_{n}}^{i_{1},\ldots,i_{n}} Y_{t}^{i}(\theta)=
\sum_{p=1}^{n} \hat{\alpha}^{i,l}_{i_{p},i_{1}\ldots,\check{\imath}_p,\ldots, n}(r_{i_{p}},r_1,\ldots,\check{r}_{p},\ldots,r_{n};\theta) \\
+ \int_{r_{1}\vee \cdots \vee r_{n}}^{t} \hat{\beta}^{i,l}_{i_{1},\ldots,i_{n}} (s;r_{1},\ldots,r_{n};\theta) \;ds
+\sum_{l=1}^{d} \int_{r_{1}\vee \cdots \vee r_{n}}^{t}  \hat{\alpha}^{i,l}_{l,i_{1},\ldots,i_{n}} (s;r_{1},\ldots,r_{n};\theta) \;dB_{s}^{l},
\end{multline*}
where $\hat{\alpha}^{i,l}_{j,i_{1},\ldots,i_{n}}=\nabla_{l}\alpha^{i}_{j,i_{1},\ldots,i_{n}}$ and $\hat{\beta}^{i,l}_{j,i_{1},\ldots,i_{n}}=\nabla_{l} \beta ^{i}_{i_{1},\ldots,i_{n}}$. More specifically, $\hat{\beta}^{i,p}_{j,i_{1},\ldots,i_{n}}$ is defined recursively by
\begin{align*}
&\hat{\beta}^{i,p}_{i_{1},\ldots,i_{n}} (s;r_{1},\ldots,r_{n};\theta) \\ 
& =\sum_{I_{1} \cup \ldots \cup I_{\nu}} \sum_{k_{1},\ldots,k_{\nu}=1 }^{m}
\Big\{ \nabla_{p} [\partial_{k_{1}\ldots k_{\nu}}^{\nu}  \mu^{i}(Y_{s}(\theta);\theta)]\;
D_{r(I_{1})}^{i(I_{1})} Y_{s}^{k_{1}}(\theta) \cdots D_{r(I_{\nu})}^{i(I_{\nu})} Y_{s}^{k_{\nu}}(\theta) \\
&\quad+\partial_{k_{1}\ldots k_{\nu}}^{\nu}  \mu^{i}(Y_{s}(\theta);\theta)\;
\sum_{p=1}^{\nu} \nabla_{p}D_{r(I_{p})}^{i(I_{p})} Y_{s}^{k_{p}}(\theta)  \, D_{r(I_{1})}^{i(I_{1})} Y_{s}^{k_{1}}(\theta) \cdots D_{\check{r}(I_p)}^{\check{\imath}(I_p)}Y_s^{\check{k}_p}(\te)
  \cdots  D_{r(I_{\nu})}^{i(I_{\nu})} Y_{s}^{k_{\nu}}(\theta) 
  \Big\},
\end{align*}
where we have set 
$$
\nabla_{p} [\partial_{k_{1}\ldots k_{\nu}}^{\nu}  \mu^{i}(Y_{s}(\theta);\theta)] = \nabla_{p}\partial_{k_{1}\ldots k_{\nu}}^{\nu}  \mu^{i}(Y_{s}(\theta);\theta) + \partial  \partial_{k_{1}\ldots k_{\nu}}^{\nu}  \mu^{i}(Y_{s}(\theta);\theta) \nabla_{p} Y_{s}(\theta).
$$ 
Notice that the same kind of equation (skipped here for sake of conciseness) holds true for the coefficients $\hat{\alpha}^{i,l}_{j,i_{1},\ldots,i_{n}}$.
\end{lemma}

The next object we need for our calculations is the inverse of the Malliavin matrix $\ga_{Y_t(\te)}$ of $Y_t(\te)$. Recall that according to (\ref{eq:def-mall-matrix}), the Malliavin matrix of $Y_t(\te)$ is defined by
 \begin{equation}\label{eq:def-gamma-Y-t}
\ga_t(\te):=\gamma_{Y_{t}(\theta)} 
= \lp \lla D_{\cdot}Y_{t}^{i}(\theta), D_{\cdot}Y_{t}^{j}(\theta) \rra \rp_{1\leq i,j \leq m},
\end{equation}
where we have set $\ga_t(\theta):=\gamma_{Y_{t}(\theta)}$ for notational sake in the computations below. We shall now compute $\ga_t^{-1}(\theta)$ as the solution to a SDE:
\begin{proposition}
The matrix valued process $\ga_t^{-1}(\theta)$ is the unique solution to the following linear equation in $\eta$:
\begin{eqnarray}
\eta_{t}(\theta) &=& \tilde{\alpha}_{0}^{-1}(Y_{t}(\theta);\theta) - \sum_{l=1}^{d}\int_{0}^{t} [ \eta_{u}(\theta) \tilde{\alpha}_{l}(Y_{u}(\theta);\theta) + \tilde{\alpha}_{l}^{T}(Y_{u}(\theta);\theta)\eta_{u} ] dB_{u}^{l} \notag\\
&&  - \int_{0}^{t} [\eta_{u}(\theta) \tilde{\beta}(Y_{u}(\theta);\theta) + \tilde{\beta}^{T}(Y_{u}(\theta);\theta)\eta_{u}(\theta) ] du,\label{eq:sde-ga-inverse}
\end{eqnarray}
with
\[ 
\tilde{\alpha}_{0}(Y_{t}(\theta);\theta) = \sum_{j=1}^{m} \int_{0}^{t}  \int_{0}^{t} \sigma^{ij}(Y_{r}(\theta);\theta) \sigma^{i'j}(Y_{r'}(\theta);\theta)\; |r-r'|^{2H-2} dr\;dr' , i,i' = 1,\ldots, m
\]
and where the other coefficients $\tilde{\al}$ and $\tilde{\beta}$ are defined by
 \begin{equation}\label{eq:def-alpha-beta}
\tilde{\alpha}_{l}(Y_{u}(\theta);\theta) = \Bigl(\partial_{k} \sigma^{i'l} (Y_{u}(\theta);\theta) \Bigr)_{1\leq i', k \leq m} \;\; \text{ and } \;\; \tilde{\beta}(Y_{u}(\theta);\theta)= \Bigl(\partial_{k} \mu^{i'} (Y_{u}(\theta);\theta) \Bigr)_{1\leq i', k \leq m}.
 \end{equation}
\end{proposition}

\begin{proof}
The proof of this fact is an adaptation of \cite[Theorem 7]{HN} to the case of a SDE with drift. We include it here for sake of completeness, and we drop the dependence of $Y$ on $\theta$ for notational sake in the computations below.

\smallskip

Let us start by invoking Proposition \ref{prop:diff-calc-rule} and equation (\ref{eq:higher-deriv-Y-t}) in order to compute the product of two first-order Malliavin derivatives:
\begin{eqnarray}\label{eq:prd-DYi-DYi'}
&& D_{r}^{j}Y_{t}^{i}\;D_{r'}^{j}Y_{t}^{i'} = \sigma^{ij}(Y_{r}) \sigma^{i'j}(Y_{r'}) + \\
&& + \sum_{k=1}^{m} \Biggl\{ \int_{0}^{t}  \sum_{l=1}^{d} \biggl[ \partial_{k} \sigma^{il}(Y_{u})\; D_{r'}^{j}Y_{u}^{i'}\;  D_{r}^{j}Y_{u}^{k}  + \partial_{k} \sigma^{i'l}(Y_{u})\; D_{r}^{j}Y_{u}^{i}\;  D_{r'}^{j}Y_{u}^{k}\; \biggr]\; dB_{u}^{l}   \notag \\
&& + \int_{0}^{t} \biggl[ \partial_{k} \mu^{i}(Y_{u})\; D_{r'}^{j}Y_{u}^{i'}\;  D_{r}^{j}Y_{u}^{k}\;+  \partial_{k} \mu^{i'}(Y_{u})\; D_{r}^{j}Y_{u}^{i}\;  D_{r'}^{j}Y_{u}^{k} du  \biggr]\Biggr\}.  \notag
\end{eqnarray}
Moreover, recall that $\ga_t$ is defined by (\ref{eq:def-gamma-Y-t}).
Thus, the covariance matrix becomes
\begin{equation*}
\gamma_{t}^{i i'} =   \sum_{j=1}^{d} \lla D^{j}Y_{t}^{i} ,\, D^{j}Y_{t}^{i'}\rra_{\ch}
= c_H \sum_{j=1}^{d} \int_{0}^{t} \int_{0}^{t} D_{r}^{j}Y_{t}^{i}(\theta)\; D_{r'}^{j}Y_{t}^{i'}(\theta)\; |r-r'|^{2H-2}\;dr\;dr'.
\end{equation*}
Plugging (\ref{eq:prd-DYi-DYi'}) into this relation, we end up with the following equation for $\gamma^{i i'}$:
\begin{multline*}
\gamma_{t}^{i i'} = \tilde{\alpha}_{0}^{ii'} + \sum_{l=1}^{d} \int_{0}^{t} \sum_{k=1}^{m}\biggl[ \partial_{k} \sigma^{i\;l}(Y_{u})\; \gamma_{u}^{i'k} + \partial_{k} \sigma^{i'l}(Y_{u})\; \gamma_{u}^{ik}\biggr] dB_{u}^{l}\\
 + \int_{0}^{t} \sum_{k=1}^{m} \biggl[ \partial_{k} \mu^{i}(Y_{u})\; \gamma_{u}^{i'k} + \partial_{k} \mu^{i'}(Y_{u})\; \gamma_{u}^{i\;k} \biggr] du.
\end{multline*}
Using our notation (\ref{eq:def-alpha-beta}) and matrix product rules, we obtain that $\ga_t$ is solution to:
\begin{equation*}
\gamma_{t} = \sum_{l=1}^{d} \int_{0}^{t} ( \tilde{\alpha}_{l}(Y_{u}) \gamma_{u} + \gamma_{u} \tilde{\alpha}_{l}^{T}(Y_{u}))dB_{u}^{l}
+ \int_{0}^{t} ( \tilde{\beta}(Y_{u}) \gamma_{u} + \gamma_{u} \tilde{\beta}^{T}(Y_{u}) ) du.
\end{equation*}
Consider now $\eta$ solution to (\ref{eq:sde-ga-inverse}). Applying again Proposition \ref{prop:diff-calc-rule}, it is readily checked that $\ga_t\eta_t=\id$ for any $t\in\ott$, which ends the proof.

\end{proof}

\begin{remark}\label{rmk:Y-non-degenerate}
Gathering equation (\ref{eq:sde-ga-inverse}) and Proposition \ref{prop:moments-sdes}, it is easily seen that for any $t>0$ and $\te\in\tte$, $Y_t(\te)$ is a non degenerate random variable in the sense given at \cite[Definition 2.1.2]{N-bk}: we have $\det(\ga_t^{-1})\in L^p(\oom)$ for any $p>1$.
\end{remark}

Now that we have derived an equation for $\eta=\ga^{-1}$, an equation for the Malliavin derivative of $\eta$ is also available:
\begin{proposition}\label{prop:mall-drv-eta}
For any $l\in\{1,\ldots,q\}$ and $n\ge 1$, the process $\eta_{t}=\ga_t^{-1}$ is $n$-time differentiable in the Malliavin calculus sense. Moreover, the process $D^{i_{1},\ldots,i_{n}}\eta_{t}$ satisfies the  following equation: for $t \geq r_{1}\vee \cdots \vee r_{n}$,
\begin{multline*}
D_{r_{1},\ldots,r_{n}}^{i_{1},\ldots,i_{n}} \eta_{t}^{ij}(\theta) =  - \sum_{k_{1}=1}^{n} \sum_{k_{2}=1}^{k_{1}}
(  D_{r_{1},\ldots,r_{k_{2}}}^{i_{1},\ldots,i_{k_{2}}} \tilde{\alpha}_{0}^{-1} 
\; D_{r_{1},r_{2},\ldots,r_{k_{1}-k_{2}}}^{i_{1},i_{2},\ldots,i_{k_{1}-k_{2}}} \tilde{\alpha}_{0}\;  
 D_{r_{1},\ldots,r_{n-k_{1}}}^{i_{1},\ldots,i_{n-k_{1}}} \tilde{\alpha}_{0}^{-1} )^{ij} \\
 - \sum_{\ell=1}^{d} \int_{r_{1}\vee \ldots  \vee r_{n}}^{t} C_{\ell,i_{1},\ldots,i_{n}}^{ij} (s;r_{1},\ldots,r_{n};\theta)  dB^{\ell}_{s}
-  \int_{r_{1}\vee \ldots  \vee r_{n}}^{t} A_{i_{1},\ldots,i_{n}}^{ij} (s;r_{1},\ldots,r_{n};\theta)  ds,
\end{multline*}
where
\begin{eqnarray*}
&& A_{i_{1},\ldots,i_{n}}^{ij} (s;r_{1},\ldots,r_{n};\theta)= \sum \sum_{k_{1},\ldots,k_{\nu}}^{m} \sum_{k=1}^{m}\\ 
&& \Big\{ 
 [ \partial_{k_{1},\ldots,k_{\nu}}^{\nu} (\tilde{\beta}(Y_{u}(\theta);\theta))^{kj}\; D_{r(I_{1})}^{i(I_{1})}\eta_{s}^{ik}(\theta) \ldots D_{r(I_{\nu})}^{i(I_{\nu})}\eta_{s}^{ik}(\theta)\;\;D_{r(I_{1})}^{i(I_{1})}Y_{s}^{i}(\theta) \ldots D_{r(I_{\nu})}^{i(I_{\nu})}Y_{s}^{i}(\theta) ] \\
&& + 
[ \partial_{k_{1},\ldots,k_{\nu}}^{\nu} (\tilde{\beta}(Y_{u}(\theta);\theta))^{ik}\; D_{r(I_{1})}^{i(I_{1})}\eta_{s}^{kj}(\theta) \ldots D_{r(I_{\nu})}^{i(I_{\nu})}\eta_{s}^{kj}(\theta)\;\; D_{r(I_{1})}^{i(I_{1})}Y_{s}^{j}(\theta) \ldots D_{r(I_{\nu})}^{i(I_{\nu})}Y_{s}^{j}(\theta) ] 
\Big\}
\end{eqnarray*}
and the same kind of equation holds for $C_{l, i_{1},\ldots,i_{n}}^{ij} (s;r_{1},\ldots,r_{n};\theta)$, with the coefficients $\beta$ replaced by $\alpha_l$.
\end{proposition}

\begin{proof}
The proof of this proposition is based on Lemma \ref{lem:higher-mall-deriv-Y} and the fact that $\frac{dA_{\lambda}^{-1}}{d\lambda} = -A_{\lambda}^{-1} \frac{dA_{\lambda}}{d\lambda} A_{\lambda}^{-1}$.

\end{proof}

Finally, one can also differentiate $\eta$ with respect to our standing parameter $\te$, which yields:
\begin{lemma}\label{lem:mall-drv-nabla-eta}
The derivative  of the inverse of the Malliavin matrix $\eta_{t}$ with respect to $\theta$ satisfies the following SDE
\begin{eqnarray*}
&&\nabla_{l} \eta_{t}(\theta)=\nabla_{l} \tilde{\alpha}_{0}^{-1}  - \sum_{\ell=1}^{d} \int_{0}^{t} \{\nabla_{l} \eta_{u}(\theta) \tilde{\alpha}_{\ell} (Y_{u}(\theta);\theta)  +\eta_{u}(\theta) \nabla_{l}[\tilde{\alpha}_{\ell}(Y_{u}(\theta);\theta)] \\
&&+ 
\nabla_{l} [\tilde{\alpha}_{\ell}^{T}(Y_{u}(\theta);\theta)] \eta_{u}(\theta) + \tilde{\alpha}_{\ell}^{T}(Y_{u}(\theta);\theta) \nabla_{l} \eta_{u}(\theta)\} dB_{u}^{\ell}- \int_{0}^{t} \{\nabla_{l} \eta_{u}(\theta) \tilde{\beta}(Y_{u}(\theta);\theta) \\
&&+ \eta_{u}(\theta) \nabla_{l} [\tilde{\beta}(Y_{u}(\theta);\theta)]  +\nabla_{l} [\tilde{\beta}^{T}(Y_{u}(\theta);\theta)] \eta_{u}(\theta) + \tilde{\beta}^{T}(Y_{u}(\theta);\theta) \nabla_{l} \eta_{u}(\theta)\} du,
\end{eqnarray*}
where $\nabla_{l}[\tilde{\beta}_{\ell}(Y_{u})] = \partial \tilde{\beta}_{\ell}(Y_{u}) \nabla_{l}Y_{u} + \nabla_{l}\tilde{\beta}_{\ell}(Y_{u})$
and   $\nabla_{l}[\tilde{\alpha}_{\ell}(Y_{u})] = \partial \tilde{\alpha}_{\ell}(Y_{u}) \nabla_{l}Y_{u} + \nabla_{l}\tilde{\alpha}_{\ell}(Y_{u})$.
\end{lemma}

\smallskip

\subsection{Probabilistic representation of the likelihood}
We have chosen to represent the log-likelihood of our sample thanks to the following formula borrowed from the stochastic analysis literature:
\begin{proposition}\label{prop:def-H-j}
Let $F$ be a $\R^m$-valued non degenerate random variable (see Remark~\ref{rmk:Y-non-degenerate} for references on this concept), and let $f$ be the density of $F$. For $n\ge 1$ and $(j_1,\ldots,j_n)\in\{1,\ldots,m\}^{n}$, let $H_{(j_1,\ldots,j_n)}(F)$ be defined recursively by $H_{(j_1)}(F)=\sum_{j=1}^{m} \delta(  (\ga_F^{-1})^{j_1j} DF^j)$ and
\begin{equation}\label{eq:H-F}
H_{(j_1,\ldots,j_n)}(F)=\sum_{j=1}^{m} \delta\lp  \lp\ga_F^{-1}\rp^{j_nj} DF^j H_{(j_1,\ldots,j_{n-1})}(F)\rp,
\end{equation}
where the Skorohod operator $\delta$ is defined at Section \ref{sec:sko-integrals}. Then one can write
\begin{equation}\label{eq:expr-density}
f(x)= \be\lc \1_{(F>x)} H_{(1,\ldots, m)}(F)\rc
= \be\lc \lp  F-x\rp_+ H_{(1,\ldots, m,1,\ldots, m)}(F)\rc,
\end{equation}
where $\1_{(F>x)}:=\prod_{i=1}^{m} \1_{(F^{i}>x_i)}$ and $(F-x)_+:=\prod_{i=1}^{m}(F^i-x_i)_+$.
\end{proposition}

\begin{proof}
The first formula is a direct application of \cite[Proposition 2.1.5]{N-bk}. The second one is obtained along the same lines, integrating by parts $m$ additional times with respect to the first one.

\end{proof}

The formula above can obviously be applied to $Y_t(\te)$ for any strictly positive $t$, since we have noticed at Remark \ref{rmk:Y-non-degenerate} that $Y_t(\te)$ is a non-degenerate random variable.
However, the expression of $H_{(j_1,\ldots,j_n)}(Y_t(\te))$ given by (\ref{eq:H-F}) is written in terms of Skorohod integrals, which are not amenable to numerical computations. We will thus recast this expression in terms of Young integrals plus some correction terms:
\begin{proposition}\label{prop:def-H-j-young}
Under Hypothesis \ref{hyp:coeff-sde} and \ref{hyp:coeff-sde2}, let us define $Q_{st}^{pji}:=(\ga_s^{-1})^{pj} D_s^{i}Y_t^{j}(\te)$ for $0\le s<t\le T$, $p,j\in\{1,\ldots,m\}$ and $i\in \{1,\ldots, d\}$. Consider  $p\in\{1,\ldots,m\}$ and a real valued random variable $G$ which is smooth in the Malliavin calculus sense. Set 
\begin{equation}\label{eq:correc-sko-strato}
U_p(G)=\sum_{i=1}^{m} \sum_{j=1}^{d} G \int_0^{t} Q_{st}^{pji} \, dB_s^{i}
- c_{H} \sum_{i=1}^{m} \sum_{j=1}^{d} \int_0^{t} \int_0^{t} D_s^{i} \lc  G Q_{rt}^{pji}\rc |r-s|^{2H-2} dr ds,
\end{equation}
where the integral with respect to $B$ is understood in the Young sense. Then the quantities $H_{(j_1,\ldots, j_n)}(Y_t(\te))$ defined at Proposition \ref{prop:def-H-j} can be expressed as
\begin{equation}\label{eq:H-j-composition-U-j}
H_{(j_1,\ldots, j_n)}(Y_t(\te))=\sum_{j=1}^{m}
U_{j_{n}}\circ\cdots\circ U_{j_{1}}\lp Y_t^{j}(\te) \rp.
\end{equation} 
\end{proposition}

\begin{proof}

It is an immediate consequence of Proposition \ref{prop:dom-strato}, since we have noticed in our Remark \ref{rmk:Y-non-degenerate} that $Y_t(\te)$ is a non-degenerate random variable.

\end{proof}

The previous proposition is still not sufficient to warranty an effective computation of the log-likelihood. Indeed, the right hand side of (\ref{eq:correc-sko-strato}) contains terms of the form $D_s [  G Q_{rt}^{pji}]$, which should be given in a more explicit form. This is the content of our next proposition.
\begin{proposition}\label{prop:comput-H}
Set $H_{(j_1,\ldots, j_n)}(Y_t(\te)):=K_{j_{1}\ldots j_{n}}$. Then the term $D_s [  K_{j_{1}\ldots j_{n}} Q_{rt}^{pji}]$  in~(\ref{eq:correc-sko-strato}) can be computed inductively as follows: 

\smallskip

\noindent
\emph{(i)}
We have $D_s [  K_{j_{1}\ldots j_{n}} Q_{rt}^{pji}]=D_s K_{j_{1}\ldots j_{n}} \, Q_{rt}^{pji} + K_{j_{1}\ldots j_{n}} \, D_sQ_{rt}^{pji}$, and $D_sQ_{rt}^{pji}$ is computed by invoking Proposition~\ref{prop:mall-drv-eta} for the derivative of $\ga_t^{-1}$ and Lemma \ref{lem:higher-mall-deriv-Y} for the derivative of $Y_t(\te)$. We are thus left with the computation of $D_s K_{j_{1}\ldots j_{n}}$.

\smallskip

\noindent
\emph{(ii)}
Assume now that we can compute $n-r$ Malliavin derivatives of $K_{j_{1}\ldots j_{r}}$. Notice that this condition is met for $r=0$, since $Y_t(\te)$ itself can be differentiated $n$ times in an explicit way according to Lemma \ref{lem:higher-mall-deriv-Y} again. Then for any $j_1,\ldots,j_{r+1}$ and $k\le n-r-1$, the quantity $K_{j_{1}\ldots j_{r+1}}$ can be differentiated $k$ times, with a Malliavin derivative given by
\begin{eqnarray}\label{eq:der-k-K}
D_{\rho_{1} \ldots \rho_{k}}^{i_{1},\ldots,i_{k}} K_{j_{1}\ldots j_{r+1}} &=&
 \sum_{\ell =1}^{k} D_{\rho_{1} \ldots \check{\rho}_{\ell} \ldots \rho_{k}}^{i_{1},\ldots,\check{i}_{\ell} \ldots,i_{k}} (K_{j_{1}\ldots j_{r}}\; Q_{\rho_{\ell} t}^{pji}) 
 +  \sum_{j=1}^{d}\int_{0}^{t} D_{\rho_{1} \ldots \rho_{k}}^{i_{1},\ldots,i_{k}} (K_{j_{1}\ldots j_{r}}\; Q_{s t}^{pji}) dB_{s}^{j} \notag\\
&& - c_{H} \int_{0}^{t} \int_{0}^{t} D_{r_{1} \rho_{1} \ldots \rho_{k}}^{k+1} (K_{j_{1}\ldots j_{r}}\; Q_{r_{2} t}^{pji}) |r_{1}-r_{2}|^{2H-2}dr_{1} dr_{2}.
\end{eqnarray}
\end{proposition}

\begin{proof}
We focus on the induction step (ii), the other one being straightforward: for a smooth random variable $W$, one easily gets by induction that
\begin{equation}\label{eq:der-p-delta-W}
D_{r_{1}\ldots r_{p}}^{i_{1},\ldots,i_{p}} \delta(W) = \sum_{\ell =1}^{p} D_{r_{1}\ldots \check{r}_{\ell} \ldots r_{p}}^{i_{1},\ldots,\check{i}_{\ell} \ldots,i_{p}} W_{r_\ell} + \delta(D^{i_{1},\ldots,i_{p}}_{r_{1}\ldots r_{p}} W ).
\end{equation}
Suppose we know the $n-r$ Malliavin derivatives for $U_{j_{r}} \circ \dots \circ U_{j_{1}}(F) := K_{j_{1}\ldots j_{r}}$. Recall moreover that
\begin{equation*}
K_{j_{1}\ldots j_{r+1}} = U_{j_{r+1}} ( K_{j_{r}\ldots j_{1}} ) = \delta( K_{j_{r}\ldots j_{1}}\; Q_{\cdot t})
\end{equation*}
Applying directly relation (\ref{eq:der-p-delta-W}) we thus get, for $k \leq m-1$:
\begin{equation*}
D_{\rho_{1} \ldots \rho_{k}}^{i_{1},\ldots,i_{k}} \delta(K_{j_{r}\ldots j_{1}}\; Q_{\cdot t}) =
 \sum_{\ell =1}^{k} D_{\rho_{1} \ldots \check{\rho}_{\ell} \ldots \rho_{k}}^{i_{1},\ldots,\check{i}_{\ell}, \ldots,i_{k-1}} (K_{j_{r}\ldots j_{1}}\; Q_{\rho_{\ell} t}) 
 + \delta(D_{\rho_{1} \ldots \rho_{k}}^{i_{1},\ldots,i_{k}} (K_{j_{r}\ldots j_{1}}\; Q_{\cdot t})).
\end{equation*}
Our formula (\ref{eq:der-k-K}) is now obtained by applying Proposition \ref{prop:dom-strato} to the Skorohod integral $\delta(D_{\rho_{1} \ldots \rho_{k}}^{k} (K_{j_{r}\ldots j_{1}}\; Q_{\cdot t}))$ above.

\end{proof}

\begin{example}
As an illustration of the proposition above, we compute $U_{2}\circ U_{1}(F)$ for  $F=Y_{t}^{i}$, $i\in\{1,\ldots,m\}$ and our $d$-dimensional fBm $B$.

\smallskip

Write first $U_{1}(Y_{t}^{i}) = \delta(Y_{t}^{i}\;(\gamma^{-1})^{1j_{1}} \; D^{j_{1}}Y_{t}^{i})$, and since this quantity has to be expressed in a suitable way for numerical approximations, we have
\begin{equation*}
U_{1}(Y_{t}^{i})  
=\sum_{j_{1}=1}^{d} Y_{t}^{i} \int_{0}^{t}  Q_{ut}^{1ij_{1}} dB_{u}^{j_{1}} - c_{H} \sum_{j_{1}=1}^{d}  \int_{0}^{t} \int_{0}^{t} D_{u_{1}}^{j_{1}} [Y_{t}^{i} Q_{u_{2}t}^{1ij_{1}}] |u_{1}-u_{2}|^{2H-2} du_{1} du_{2},
\end{equation*}
where $Q$ is defined at Proposition \ref{prop:def-H-j-young} and where the first integral in the right hand side is understood in the Young sense. In order to compute the second one, we have to compute Malliavin derivatives. This is done through Lemma \ref{lem:higher-mall-deriv-Y} for $Y$ and Proposition \ref{prop:mall-drv-eta} for $Q$.

\smallskip

We now have to differentiate $U_{1}(Y_{t}^{i})$: the derivation rules for Skorohod integrals immediately yield
\begin{equation*}
D_{u_{2}}^{j_{2}}[U_{1}(Y_{t}^{i})] = \sum_{j_{2}=1}^{d}Y_{t}^{i}Q_{u_2t}^{2ij_{2}} + \sum_{j_{2}=1}^{d}\delta(D_{u_{2}}^{j_{2}}Y_{t}^{i} Q_{\cdot t}^{2ij_{2}}).
\end{equation*}
Once again, the Skorohod integral above is not suitable for numerical approximations. Write thus
\begin{multline*}
D_{u_{2}}^{j_{2}}[U_{1}(Y_{t}^{i})]=
\sum_{j_{2}=1}^{d} Y_{t}^{i}Q_{u_{2}t}^{2ij_{2}} + \sum_{j_{2}=1}^{d}\int_{0}^{t} D_{u_{2}}^{j_{2}}[Y_{t}^{i}Q_{rt}^{2ij_{2}}] dB_{r}^{j_{2}} \\
 - c_{H} \sum_{j_{2}=1}^{d}\int_{0}^{t} \int_{0}^{t} D_{u_{2}}^{j_{2}}D_{u_{1}}^{j_{1}} [Y_{t}^{i} Q_{u_{2}t}^{2ij_{2}}] |u_{2}-u_{1}|^{2H-2} du_{1}du_{2},
\end{multline*}
and compute the Malliavin derivatives of the products $YQ$ thanks to Lemma \ref{lem:higher-mall-deriv-Y} for $Y$ and Proposition \ref{prop:mall-drv-eta} for $Q$. Once this is done, just write
\begin{eqnarray*}
&& U_{2}(U_{1}(Y_{t}^{i})) = \delta(U_{1}(Y_{t}^{i}) Q_{\cdot t}^{i})\\
&=& \sum_{j_{2}=1}^{d}U_{1}(Y_{t}^{i}) \int_{0}^{t} Q_{ut}^{2ij_{2}} dB_{u}^{j_{2}} - c_{H} \sum_{j_{2}=1}^{d}\int_{0}^{t} \int_{0}^{t} D_{u_{2}}^{j_{2}} [U_{1}(Y_{t}^{i}) Q_{u_{1}t}^{2ij_{2}}] |u_{2} - u_{1}|^{2H-2} du_{1}du_{2}.
\end{eqnarray*}
\end{example}

\smallskip

In order to give our formula for the derivative of the log-likelihood, we still need to compute the derivative with respect to $\theta$ of $H_{(j_1,\ldots, j_n)}(Y_{t}(\te))$. For this we state the following lemma
\begin{lemma}\label{lem:nabla-U-p}
The derivative with respect to $\theta$ of $U_{p}(Y_{t}(\theta))$ can be written as
\begin{eqnarray*}
\nabla_{l}U_p(Y_{t}^{i}(\theta))&=&\sum_{j=1}^{d} [\nabla_{l}Y_{t}^{i}(\theta) \int_0^{t} Q_{st}^{pij}(\theta)\, dB_s^{j} + Y_{t}^{i}(\theta) \int_0^{t} \nabla_{l}[Q_{st}^{pij}(\theta)]\, dB_s^{j}]\\
&& - c_{H}\sum_{j=1}^{d} \int_0^{t} \int_0^{t} \nabla_{l} [D_s^{j}  Y_{t}^{i}(\theta) Q_{rt}^{pij}(\theta)] |r-s|^{2H-2} dr ds,
\end{eqnarray*}
where $\nabla_{l}Y_{t}^{i}(\theta)$ is computed according to Proposition \ref{prop:dif-theta-Y} and $ \nabla_{l} [D_s^{j}  Y_{t}^{i}]$ is given by Lem\-ma~\ref{lem:higher-mall-deriv-nabla-Y}. As far as $\nabla_{l}[Q_{st}^{pj} (\theta)]$ is concerned, it is obtained through the following equation:
\[
\nabla_{l}[Q_{st}^{pj} (\theta)]= \nabla_{l}\eta_{s}^{pj}(\theta)\;D_{s}Y_{t}^{j}(\theta)\;+\;\eta_{s}^{pj}(\theta) \nabla_{l}[D_{s}Y_{t}^{j}(\theta)],
\]
where the expression for $\nabla_{l}\eta_{s}^{pj}(\theta)$ is a consequence of Lemma \ref{lem:mall-drv-nabla-eta}.
\end{lemma}
\smallskip

We are now ready to state our probabilistic expression for the log-likelihood function~(\ref{eq:lln-rough}).

\begin{theorem}\label{prop:loglik}
Assume Hypothesis \ref{hyp:coeff-sde} and \ref{hyp:coeff-sde2} hold true. Let  $y_{t_{i}}$, $i=1,\ldots,n$ be the observation arriving at time $t_{i}$. Let also $Y_{t_{i}}$ be the solution to the SDE (\ref{eq:sde}) at time $t_{i}$. Then, the gradient of the log-likelihood function admits the following probabilistic representation:
$\nabla_{l}\ell_n(\te) =\sum_{i=1}^{n}\frac{V_i(\te)}{W_i(\te)}$, with
\begin{equation}\label{eq:expr-W-i}
W_i(\te)=\mathbf{E} \biggl[ \mathbf{1}_{(Y_{t_{i}}(\theta)>y_{t_{i}})}\; H_{(1,\ldots,m)}\Bigl( Y_{t_{i}}(\theta) \Bigr) \biggr]
\end{equation}
and 
\begin{multline}\label{eq:expr-V-i}
V_i(\te)=
\mathbf{E} \biggl[\nabla_{l}Y_{t_{i}}(\theta) \; \mathbf{1}_{(Y_{t_{i}}(\theta)>y_{t_{i}})} \;H_{(1,\ldots,m,1,\ldots, m)} \Bigl( Y_{t_{i}}(\theta) \Bigr)  \\
+ \Bigl( Y_{t_{i}}(\theta) - y_{t_{i}} \Bigr)_{+} \nabla_{l}H_{(1,\ldots,m,1,\ldots, m)} \Bigl( Y_{t_{i}}(\theta) \Bigr) \biggr],
\end{multline}
where  \emph{(i)} $H_{(j_1,\ldots,j_n)}( Y_{t_{i}}(\theta))$ is given recursively by (\ref{eq:H-j-composition-U-j}) and computed at Proposition \ref{prop:comput-H} \emph{(ii)} $\nabla_{l}Y_{t_{i}}(\theta)$ is given by Proposition \ref{prop:dif-theta-Y} \emph{(iii)} $\nabla_{l}H_{(1,\ldots,m,1,\ldots, m)}$ is obtained by applying~Lem\-ma \ref{lem:nabla-U-p}.
\end{theorem}

\begin{proof}
Recall that under Hypothesis \ref{hyp:coeff-sde} and \ref{hyp:coeff-sde2}, $Y_t(\te)$ admits a $\cac^\infty$ density $f(t,\cdot;\, \te)$ for any $t>0$ and $\te\in\tte$. Moreover, we have defined $\ell_n(\te)$ as $\ell_n(\te)=\sum_{i=1}^{n} \ln(f(t_i,y_{t_i};\, \te))$. Thus
\begin{equation*}
\nabla_{l}\ell_n(\te)=\sum_{i=1}^{n}\frac{\nabla_l f(t_i,y_{t_i};\, \te)}{f(t_i,y_{t_i};\, \te)}
:=\sum_{i=1}^{n}\frac{V_i(\te)}{W_i(\te)}.
\end{equation*}
Now $W_i(\te)$ can be expressed like (\ref{eq:expr-W-i}) by a direct application of (\ref{eq:expr-density}), first relation. As far as $V_i(\te)$ is concerned, write 
\begin{equation*}
f(t_i,y_{t_i};\, \te)= 
\be\lc \lp  Y_{t_i}(\te)-y_{t_i}\rp_+ H_{(1,\ldots, m,1,\ldots, m)}(Y_{t_i}(\te))\rc,
\end{equation*}
according to the second relation in (\ref{eq:expr-density}). By using standard arguments, one is allowed to differentiate this expression within the expectation, which directly yields (\ref{eq:expr-V-i}).

\end{proof}

\section{Discretization of the log-likelihood}
\label{sec:discret-likelihood}

The expression of the log-likelihood that we derived in Proposition \ref{prop:loglik} is a fraction of two expectations that do not have explicit formulas even in the one-dimensional case. In addition, our goal is to find the root of this non-explicit expression, the ML estimator, which is an even  harder task. To solve this problem in practice we first use a stochastic approximation algorithm in order to find the root of $\nabla_{l} \ell_{n}(\theta)$. In each iteration of the algorithm we compute the value of the expression using Monte-Carlo (MC) simulations. For each Monte-Carlo simulation, since we do not have available an exact way of simulating the kernels of the expectation, we use an Euler approximation scheme. More specifically, we simulate using Euler approximation terms such as $Y_{t}$, $DY_{t}$, which are solutions to fractional stochastic differential equations.

Therefore, in our approach we have three types of error in the computation of the MLE: the error of the stochastic approximation algorithm, the Monte-Carlo error and the discretization bias introduced by the Euler approximation for the stochastic differential equations. Our aim here is to combine the Monte Carlo and Euler approximations in an optimal way in order to get a global error bound for the computation of $\nabla_{l} \ell_{n}(\theta)$.

\subsection{Pathwise convergence of the Euler scheme}

The Euler scheme is the main source of error in our computations. There is always a trade-off between the number of Euler steps and the number of simulations, but what is usually computationally costly is the number of Euler steps. This is even worse when we deal with fractional SDEs, since the rate of convergence depends on $H$ and the closer the value of $H$ to $1/2$, the more steps are required for the simulation.

In this section, we compute the magnitude of the discretization error we introduce. We measure the bias of the Euler scheme via the root mean square error. That is,  we want to estimate the quantity $\sup_{\tau\in[0,T]}( \mathbf{E} |Y_{\tau}(\theta) - \bar{Y}_{\tau}^{M}(\theta) |^{2})^{1/2}$,
where $Y_t(\theta)$ is the solution to the SDE (\ref{eq:sde}) and $\bar{Y}_{\tau}^{M}(\theta)$ is the Euler approximation of $Y_{\tau}(\theta)$ given on the grid $\{\tau_k;\, k\le M\}$ by
\begin{equation}\label{eq:euler}
\bar{Y}_{\tau_{k+1}}^{M} (\theta)= \bar{Y}_{\tau_{k}}^{M}(\theta) + \mu(\bar{Y}_{\tau_{k}}^{M}(\theta);\theta) (\tau_{k+1}-\tau_{k}) + \sum_{j=1}^{d} \sigma^{j}(\bar{Y}_{\tau_{k}}^{M}(\theta);\theta) \delta B^{M,j}_{\tau_{k} \tau_{k+1}},
\end{equation}
in which we denote $\delta B^{M,j}_{\tau_{k} \tau_{k+1}} =  B_{\tau_{k+1}}^{M,j} - B_{\tau_{k}}^{M,j}$ and  $\tau_{k} = \frac{kT}{M}$ for $k = 0,\ldots, M-1$. Notice that those estimates can be found in \cite{DNT,FV,MS}. We include their proof here because it is simple enough, and also because they can be easily generalized to the case of a linear equation. This latter case is of special interest for us, since it corresponds to Malliavin derivatives, and is not included in the aforementioned references.
\begin{notation}
For simplicity, in this section we write $Y:=Y(\theta)$.
\end{notation}
\begin{proposition}\label{prop:EulerSDE}
Let $T>0$ and recall that $\bar{Y^{M}}$ is defined by equation (\ref{eq:euler}). Then, there exists a random variable $C$ with finite $L^{p}$ moments such that for all $\gamma < H$ and $H>1/2$ we have
\begin{equation}\label{eq:error1}
\|Y_t - \bar{Y} \|_{\gamma,T} \leq C_T\; M^{1-2\gamma}
\end{equation}
Consequently, we obtain that the MSE is of order ${\mathcal{O}}(M^{1-2\gamma})$.
\end{proposition}

\begin{proof}
In order to prove (\ref{eq:error1}) we apply techniques of the classical numerical analysis for the flow of an ordinary differential equation driven by a smooth path. Namely, the exact flow of (\ref{eq:sde}) is given by $\Phi(y; s, t) := Y_{t}$, where $Y_{t}$ is the unique solution of (\ref{eq:sde}) when $t\in [s,T]$ and the initial condition is $Y_{s} = y$. Introduce also the numerical flow
\begin{equation}\label{eq:numFlow}
\Psi(y; \tau_{k}, \tau_{k+1}) := y + \mu(y) (\tau_{k+1}-\tau_{k}) + \sum_{j=1}^{d} \sigma^{j}(y) \delta B^{M,j}_{\tau_{k} \tau_{k+1}},
\end{equation}
where $\tau_{k} = \frac{kT}{M}$, $k=0,\ldots,M-1$. Thus, we can write that
\begin{eqnarray*}
\bar{Y}_{\tau_{k+1}}^{M} & = & \Psi\Bigl(\bar{Y}_{\tau_{k}}^{M};\; \tau_{k}, \tau_{k+1}\Bigr),\; k=0, \ldots, M-1\\
Y_{0}^{M} & = & \alpha.
\end{eqnarray*}
For $q>k$ we also have that
\begin{equation*}
\Psi(y; \tau_{k}, \tau_{q}) := \Psi(\cdot; \tau_{q-1}, \tau_{q})\circ \Psi(\cdot; \tau_{q-2}, \tau_{q-1})\circ \ldots \circ \Psi(y; \tau_{k}, \tau_{k+1}).
\end{equation*}
The one-step error computes as
\begin{eqnarray}\label{eq:r_k}
r_{k} &=& \Phi(y; \tau_{k}, \tau_{k+1}) - \Psi(y; \tau_{k}, \tau_{k+1}) \notag\\
&=& \int_{\tau_{k}}^{\tau_{k+1}} \Bigl[ \mu(Y_{s}) - \mu(y) \Bigr] ds + \int_{\tau_{k}}^{\tau_{k+1}} \Bigl[ \sigma(Y_{s}) - \sigma(y) \Bigr] dB_{s}
\end{eqnarray}
Furthermore, since $Y\in {\mathcal{C}}^{\gamma}$ and $B \in {\mathcal{C}}^{\gamma}$ for $\gamma >1/2$, using (\ref{y}) we have
\begin{eqnarray*}
\Bigl|\int_{\tau_{k}}^{\tau_{k+1}} \Bigl[ \sigma(Y_{s}) - \sigma(y) \Bigr] dB_{s} \Bigr| 
&\leq & c_{ \gamma}\;\|\partial \sigma\|_{\infty} \|Y\|_{\gamma}\;\|B\|_{\gamma} \; \left|\frac{T}{M}\right|^{2\gamma} \\
&\leq & c_{\gamma,\sigma}\;\|\partial \sigma\|_{\infty}\; \|B\|_{\gamma}^{1/\gamma}\;\|B\|_{\gamma} \; \left|\frac{T}{M}\right|^{2 \gamma},
\end{eqnarray*}
where we used the fact that $\|Y\|_{\gamma} \leq c_{\sigma}\|B\|_{\gamma}^{1/\gamma}$ (see Proposition \ref{prop:moments-sdes}). 
Similarly, for the drift part we have
\begin{eqnarray*}
\Bigl |\int_{\tau_{k}}^{\tau_{k+1}} \Bigl[ \mu(Y_{s}) - \mu(y) \Bigr] ds \Bigr |  &\leq& c_{\gamma}\;\|\partial\mu\|_{\infty}\;\|Y\|_{\gamma} \; \left|\frac{T}{M}\right|^{\gamma+1}\\
& \leq & c_{\gamma,\mu}\;\|\partial\mu\|_{\infty}\;\|B\|_{\gamma}^{1/\gamma} \; \left|\frac{T}{M}\right|^{\gamma+1}.
\end{eqnarray*}
Therefore, the one-step error (\ref{eq:r_k}) satisfies
\begin{equation}\label{eq:bnd-r-k}
|r_{k}| \leq  c_{\mu,\sigma}\; \|B\|_{\gamma}^{1+1/\gamma}\; \left|\frac{T}{M}\right|^{2\gamma}.
\end{equation}
Now, we can write the classical decomposition of the error in terms of the exact and numerical flow. Since $\bar{Y}_{\tau_{k}}^{M} = \Phi(\bar{Y}_{\tau_k}^{M}; \tau_{k}, \tau_{k})$ and $Y_{\tau_{k}} = \Phi(\bar{Y}_{\tau_0}^{M}; \tau_{0}, \tau_{k})$ we have
\begin{equation}\label{eq:error}
\bar{Y}_{\tau_{q}}^{M} - Y_{\tau_{q}} = \Phi(\bar{Y}_{\tau_{0}}; \tau_{0}, \tau_{k}) -  \Phi(\bar{Y}_{\tau_{q}}; \tau_{q}, \tau_{q}) = \sum_{k=0}^{q-1} \Bigl(  \Phi(\bar{Y}_{\tau_{k}}^{M}; \tau_{k}, \tau_{q}) -  \Phi(\bar{Y}_{\tau_{k+1}}; \tau_{k+1}, \tau_{q}) \Bigr).
\end{equation}
Since 
$ \Phi\Bigl(\bar{Y}_{\tau_{k}}^{M}; \tau_{k}, \tau_{q}\Bigr) = \Phi\Bigl(  \Phi(\bar{Y}_{\tau_{k}}^{M}; \tau_{k}, \tau_{k+1}); \tau_{k+1}, \tau_{q}\Bigr)$
we obtain
\begin{eqnarray*}
\Bigl| \Phi(\bar{Y}_{\tau_{k}}^{M}; \tau_{k}, \tau_{q}) -  \Phi(\bar{Y}_{\tau_{k+1}}^{M}; \tau_{k+1}, \tau_{q})\Bigr| &=&
\Bigl| \Phi\Bigl(  \Phi(\bar{Y}_{\tau_{k}}^{M}; \tau_{k}, \tau_{q}); \tau_{k+1}, \tau_{q}\Bigr) -  \Phi(\bar{Y}_{\tau_{k+1}}^{M}; \tau_{k+1}, \tau_{q})\Bigr| \\
&\leq & C_{T}(\|B\|_{\gamma})\; | \Phi(\bar{Y}_{\tau_{k}}^{M}; \tau_{k}, \tau_{k+1}) - \bar{Y}_{\tau_{k+1}}^{M}|,
\end{eqnarray*}
where we have used the fact that 
\[| \Phi(\alpha; t, s) - \Phi(\beta;  t, s)\leq C_{T}(\|B\|_{\gamma}) |\alpha-\beta|,\]
where $C_{T}$ is a subexponential function (see Proposition \ref{prop:moments-sdes} again). Moreover, owing to relation (\ref{eq:bnd-r-k}),
\begin{equation}\label{eq:r_k2}
| \Phi(\bar{Y}_{\tau_{k}}^{M}; \tau_{k}, \tau_{q}) -  \bar{Y}_{\tau_{k+1}}^{M}| = | r_k | 
\leq c_{\mu,\sigma}\; \|B\|_{\gamma}^{1+1/\gamma}\; \left|\frac{T}{M}\right|^{2 \gamma}.
\end{equation}
Therefore, replacing (\ref{eq:r_k2}) in (\ref{eq:error}) for any $q \leq n$ we obtain 
\begin{eqnarray*}
| \bar{Y}_{\tau_{q}}^{M} - Y_{\tau_{q}} | &\leq &  c_{\mu,\sigma}\;  \|B\|_{\gamma}^{1+1/\gamma}\; \sum_{k=0}^{q-1} \left|\frac{T}{M}\right|^{2 \gamma}
\end{eqnarray*}
Let us push forward this analysis to H\"older type norms on the grid $0 \leq \tau_{1}<\ldots < \tau_{n}=T$. We have  for $q\geq p$
\begin{eqnarray*}
&& \delta\Bigl(Y-\bar{Y}^{M}\Bigr)_{\tau_{p} \tau_{q}} \\
&=& \Bigl( \Phi(Y_{\tau_{p}}; \tau_{p}, \tau_{q}) - Y_{\tau_{p}}\Bigr) - \Bigl( \Psi(\bar{Y}_{\tau_{p}}^{M}; \tau_{p}, \tau_{q}) - \bar{Y}_{\tau_{p}}^{n}\Bigr)\\
&=&  \Bigl( \Phi(Y_{\tau_{p}}; \tau_{p}, \tau_{q}) -Y_{\tau_{p}}\Bigr) - \Bigl( \Phi(\bar{Y}_{\tau_{p}}^{M}; \tau_{p}, \tau_{q}) - \bar{Y}_{\tau_{p}}^{M}\Bigr) - \Bigl( \Psi(\bar{Y}_{\tau_{p}}^{M}; \tau_{p}, \tau_{q}) - \Phi(\bar{Y}_{\tau_{p}}^{M}; \tau_{p}, \tau_{q})\Bigr)\\
&=&  \biggl( \Bigl( \Phi(Y_{\tau_{p}}; \tau_{p}, \tau_{q}) - \Phi(\bar{Y}_{\tau_{p}}^{M}; \tau_{p}, \tau_{q}) \Bigr) - \Bigl( Y_{\tau_{p}} - \bar{Y}_{\tau_{p}}^{M} \Bigr) \biggr) 
- \Bigl( \Psi(\bar{Y}_{\tau_{p}}^{M}; \tau_{p}, \tau_{q}) - \Phi(\bar{Y}_{\tau_{p}}^{M}; \tau_{p}, \tau_{q})\Bigr).
\end{eqnarray*}
Similar to the calculations leading to (\ref{eq:r_k2}) we obtain
\begin{equation*}
\Bigl| \Psi(\bar{Y}_{\tau_{p}}^{M}; \tau_{p}, \tau_{q}) - \Phi(\bar{Y}_{\tau_{p}}^{M}; \tau_{p}, \tau_{q})\Bigr| \leq  c_{\mu,\sigma}\;  \|B\|_{\gamma}^{1+1/\gamma}\;\sum_{k=p}^{q-1} \left|\frac{T}{M}\right|^{2\gamma}.
\end{equation*}
Moreover, owing to Proposition \ref{prop:moments-sdes} part (2), observe that
\begin{equation*}
\frac{\Bigl| \Bigl( \Phi(Y_{\tau_{p}}; \tau_{p}, \tau_{q}) - \Phi(\bar{Y}_{\tau_{p}}^{M}; \tau_{p}, \tau_{q}) \Bigr) - \Bigl( Y_{\tau_{p}} - \bar{Y}_{\tau_p}^{M} \Bigr) \bigr|}{|\tau_{q}-\tau_{p}|^{\gamma}} 
\leq c(\|B\|_{\gamma})\;|Y_{\tau_{p}}-\bar{Y}_{\tau_{p}}^{M}|.
\end{equation*}
Consequently, we have that for $0 \leq p < q \leq M$
\begin{equation*}
\Bigl|\delta\Bigl(Y-\bar{Y}^{M}\Bigr)_{\tau_{p} \tau_{q}}\Bigr| \leq  c'(\|B\|_{\gamma}^{1+1/\gamma})  \Bigl\{ \sum_{k=p}^{q-1} \left|\frac{T}{M}\right|^{2 \gamma} +\;|\tau_{q} - \tau_{p}|^{\gamma} \sum_{k=0}^{q} \left|\frac{T}{M}\right|^{2 \gamma} \Bigr\}
\end{equation*}
which easily yields that 
\begin{equation*}
\sup_{p,q=0,1,\ldots,M-1, p\neq q} \frac{\Bigl|\delta\Bigl(Y-\bar{Y}^{M}\Bigr)_{\tau_{p} \tau_{q} } \Bigr|} {|\tau_{p}-\tau_{q}|^{\gamma}} \leq c(\|B\|_{\gamma})\; M^{1-2\;\gamma}.
\end{equation*}
By ``lifting'' this error estimate to $[0,T]$ and since $|t-s| \leq T/M$, 
\begin{equation}
\|Y_t - \bar{Y} \|_{\gamma,\infty,T} \leq  C\;M^{1-2\;\gamma},
\end{equation}
which concludes the first part of the proof. 

Regarding the order of the Mean Square Error, it suffices to note that the constant $C$ has finite $L^{p}$ moments.
\end{proof}

As mentioned before, an elaboration of Proposition \ref{prop:EulerSDE} is needed in the sequel. Indeed, in the expression of the log-likelihood in Proposition \ref{prop:loglik} we need to discretize more complicated quantities of the underlying process, such as (\ref{eq:higher-deriv-Y-t}) or (\ref{eq:sde-ga-inverse}). To this aim, let us notice first that all those equations can be written under the following generic form: 
\begin{equation}\label{eq:eqZ}
Z_{t} = \alpha + \int_{0}^{t} \xi_{u}^{2} Z_{u}   du + \int_{0}^{t}  \xi_{u}^{1,j} Z_{t} dB_{u}^{j},
\end{equation}
where  $\xi^1$, $\xi^{2}$ are stochastic processes with bounded moments of any order. The corresponding Euler discretization is
\begin{equation}\label{eq:Zeuler}
\bar{Z}_{\tau_{k}}^{M} = \bar{Z}_{\tau_{k}}^{M} + \xi_{\tau_{k}}^{2} \bar{Z}_{\tau_{k}}^{M} (\tau_{k+1}-\tau_{k}) + \sum_{j=1}^{d} \xi_{\tau_{k}}^{1,j} \bar{Z}_{\tau_{k}} \; \delta B_{\tau_{k}\tau_{k+1}}^{j,M},
\end{equation}
and we give first an approximation result in this general context:
\begin{proposition}\label{prop:ZEuler}
Let $T>0$, and consider the $\R^q$-valued solution $Z$ to equation (\ref{eq:eqZ}), where $\al\in\R^q$, $\xi^{2}, \xi^{1,j}\in\R^{q,q}$ and we suppose that $\|\xi^{2}\|_{\ga}$ and $\|\xi^{1,j}\|_{\ga}$ belong to $L^p(\oom)$ for any value of $p\ge 1$. Let $\bar{Z}^{M}$ be defined by equation (\ref{eq:Zeuler}). Then, there exists a  random variable $C^{'}$ with $L^p$ finite moments, such that for all $\gamma < H$ and $H>1/2$ we have
\begin{equation}
\|Z - \bar{Z} \|_{\gamma,T} \leq C^{'}_T\; M^{1-2\gamma}
\end{equation}
Consequently, we obtain that the Mean Square Error is of order ${\mathcal{O}}(M^{1-2\gamma})$.
\end{proposition}

\begin{proof}
We follow a similar approach as in the previous proposition. Thus, the exact flow is equal to $\Phi(\zeta; s, t):= Z_{t}$, where $Z_{t}$ is the unique solution of equation (\ref{eq:eqZ}) when $t\in[s,T]$ and the initial condition is $Z_{s} = \zeta$.  Consider also the numerical flow
\begin{equation*}
\Psi(\zeta; \tau_{k}, \tau_{k+1}) := \zeta + \xi_{u}^{2} \zeta (\tau_{k+1}-\tau_{k}) + \sum_{j=1}^{d} \xi_{u}^{1,j} \zeta \delta B_{\tau_{k}\tau_{k+1}}^{j,M},
\end{equation*}
where $\tau_{k} = kT/M$, $n=0,\ldots,M-1$. Thus, we have
\begin{eqnarray*}
\bar{Z}_{\tau_{k+1}}^{M} &=& \Psi(\bar{Z}_{\tau_{k}}^{M}; \tau_{k+1}, \tau_{k}), \;k=0,\ldots,M-1\\
\bar{Z}_{0}^{M} &=& \alpha.
\end{eqnarray*}
In this case, the one-step error can be written as
\begin{eqnarray*}
r_{k} &=&  \Phi(\zeta; \tau_{k}, \tau_{k+1}) - \Psi(\zeta; \tau_{k}, \tau_{k+1})\\
&=& \int_{\tau_{k}}^{\tau_{k+1}} \xi_{u}^{2} (Z_{s} - \zeta) du + \int_{\tau_{k}}^{\tau_{k+1}} \xi_{u}^{1} (Z_{s} - \zeta) dB_{u}
\end{eqnarray*}
We now treat each term separately. Therefore, using the fact that $\|Z\|_{\ga} \leq \exp(c\|B\|_{\gamma}^{1/\gamma}) $, which is recalled at Proposition \ref{prop:moments-sdes} point (4) in a slightly different context, we have that
\begin{eqnarray*}
\Bigl|\int_{\tau_{k}}^{\tau_{k+1}}  \xi_{s}^{1}(Z_{s} - \zeta) dB_{s} \Bigr| 
&\leq & c_{ \gamma}\; \|Z \xi^{1}\|_{\gamma}\;\|B\|_{\gamma} \; \left|\frac{T}{M}\right|^{2\gamma} \\
&\leq & c_{\gamma}\; \|\exp(\|B\|_{\gamma}^{1/\gamma})\;\|B\|_{\gamma} \; \left|\frac{T}{M}\right|^{2 \gamma}.
\end{eqnarray*}
Similarly, we also have
\begin{eqnarray*}
\Bigl|\int_{\tau_{k}}^{\tau_{k+1}}  \xi_{s}^{2}(Z_{s} - \zeta) d{s} \Bigr| 
&\leq & c_{ \gamma}\; \|Z \xi^{2}\|_{\gamma}\;\|B\|_{\gamma} \; \left|\frac{T}{M}\right|^{2\gamma} \\
&\leq & c_{\gamma}\; \|\exp(\|B\|_{\gamma}^{1/\gamma})\;\|B\|_{\gamma} \; \left|\frac{T}{M}\right|^{2 \gamma}.
\end{eqnarray*}
Therefore, the one-step error satisfies the following inequality
\begin{equation*}
|r_{k}| \leq c_{\gamma} \exp(\|B\|_{\gamma}^{1/\gamma})\;\|B\|_{\gamma} \; \left|\frac{T}{M}\right|^{2 \gamma}.
\end{equation*}

Along the same lines as for Proposition \ref{prop:EulerSDE}, the decomposition of the error in terms of the exact and numerical flow becomes
\begin{equation*}
\bar{Z}^{M}_{\tau_{q}} - Z_{\tau_{q}} = \Phi(\bar{Z}^{M}_{\tau_{q}}; \tau_{q}, \tau_{q})-\Phi(\bar{Z}^{M}_{\tau_{0}}; \tau_{0}, \tau_{k}) = 
\sum_{k=0}^{q-1} 
\Bigl(  \Phi(\bar{Z}^{M}_{\tau_{k+1}}; \tau_{k+1}, \tau_{q})-\Phi(\bar{Z}_{\tau_{k}}^{M}; \tau_{k}, \tau_{q})    \Bigr),
\end{equation*}
and the same inequalities allowing to go from (\ref{eq:error}) to (\ref{eq:r_k2}) yield
\begin{eqnarray*}
| \bar{Z}_{\tau_{q}} - Z_{\tau_{q}} | &\leq &  c_{\gamma}\;  \left| \frac{T}{M}\right|^{2 \gamma}.
\end{eqnarray*}
The claim of the proposition follows now as in Proposition \ref{prop:EulerSDE}.

\end{proof}

We now use the previous proposition in order to approximate the kernels of the expectations in $\nabla_{l}\ell_{n}(\theta)$. Let us first introduce the following notation:
\begin{notation}
Let $W_{i}(\theta)$, $V_{i}(\theta)$ as in (\ref{eq:expr-W-i}) and (\ref{eq:expr-V-i}) respectively and define $w_{i}(\theta)$ and $v_{i}(\theta)$ as
\begin{eqnarray}
w_i(\te)&=&\mathbf{1}_{(Y_{t_{i}}(\theta)>y_{t_{i}})}\; H_{(1,\ldots,m)}\Bigl( Y_{t_{i}}(\theta) \Bigr) \label{eq:expr-w-i}\\
v_i(\te)&=& \nabla_{l}Y_{t_{i}}(\theta) \; \mathbf{1}_{(Y_{t_{i}}(\theta)>y_{t_{i}})} \;H_{(1,.,m,1,., m)} 
+ \Bigl( Y_{t_{i}}(\theta) - y_{t_{i}} \Bigr)_{+} \nabla_{l}H_{(1,.,m,1,., m)}. \label{eq:expr-v-i}
\end{eqnarray}
Let also $\bar{w}_{i}^{M}$ and $\bar{v}_{i}^{M}$ to be the Euler discretized versions of (\ref{eq:expr-w-i}) and (\ref{eq:expr-v-i}) using \ref{prop:ZEuler}, and set $\bar W_{i}^{M}(\theta)=\be[\bar{w}_{i}^{M}]$ and $\bar V_{i}^{M}(\theta)=\be[\bar{v}_{i}^{M}]$.
\end{notation}
Our convergence result for $\nabla_{l}\ell_{n}(\theta)$ can be read as follows:
\begin{theorem}\label{thm:cvgce-Vi-Wi}
Recall from Theorem \ref{prop:loglik} that $\nabla_{l}\ell_n(\te)$ can be decomposed as $\nabla_{l}\ell_n(\te) =\sum_{i=1}^{n}\frac{V_i(\te)}{W_i(\te)}$. Then the following approximation result holds true:
\begin{equation*}
\lln V_i(\te)- \bar V_{i}^{M}(\theta)\rrn + \lln W_i(\te)- \bar W_{i}^{M}(\theta)\rrn
\le \frac{c}{M^{2\ga -1}},
\end{equation*}
for a strictly positive constant $c$.
\end{theorem}

\begin{proof}
We focus on the bound for $| V_i(\te)- \bar V_{i}^{M}(\theta)|$, the other one being very similar. Now, applying Proposition \ref{prop:ZEuler} to the particular case of the equations governing Malliavin derivatives, we easily get
\begin{equation*}
\|v_{t} - \bar{v}\|_{\gamma, T} \le  C_{2} M^{1-2\gamma},
\end{equation*}
for an integrable random variable $C_2$. The proof is now easily finished by invoking the inequality
\begin{equation*}
\lln V_i(\te)- \bar V_{i}^{M}(\theta)\rrn \le \be \lc  \|v_{t} - \bar{v}\|_{\gamma, T} \rc.
\end{equation*}

\end{proof}

\begin{remark}
We have given two separate approximations for $V_i(\te)$ and $W_i(\te)$. In order to fully estimate $(V_i(\te)/W_i(\te))-(\bar V_{i}^{M}(\theta)/\bar W_{i}^{M}(\theta))$, one should also prove that $W_i(\te)$ is bounded away from 0. This requires a lower bound for densities of differential equations driven by fractional Broawnian motion, which are out of the scope of the current article.

\end{remark}

\subsection{Efficiency of the Monte Carlo simulation}

In this section we aim to study the computational tradeoff between  the length of a time period in the Euler discretization (i.e. $1/M$) and the number of Monte Carlo simulations of the sample path (i.e. $N$). In order to do so we consider $\bar{w}_{i}^{M}$ and $\bar{v}_{i}^{M}$ as above. 

Recall that, given $t$ units of computer time, the Monte-Carlo estimators for $W_{i}(\theta)$ and $V_{i}(\theta)$ can be written as
\begin{equation*}
\frac{1}{c_{1}(t,\frac{1}{M})} \sum_{k=1}^{c_{1}(t,\frac{1}{M})} {w}_{i,k}^{M},\;\;\frac{1}{c_{2}(t,\frac{1}{M})} \sum_{k=1}^{c_{2}(t,\frac{1}{M})} {v}_{i,k}^{M}
\end{equation*}
where $\{{w}_{i, \ell}^{M}; \, \ell\ge 1\}$ (resp. $\{{v}_{i, \ell}^{M}; \, \ell\ge 1\}$) is a sequence of i.i.d. copies of ${w}_{i}^{M}$ (resp. of ${v}_{i}^{M}$), and $c_{1}(t,\frac{1}{M}),c_{2}(t,\frac{1}{M})$ are the maximal number of runs one is allowed to consider with $t$ units of computer time.  Using the result by \cite{DG} we can state the following proposition:
\begin{proposition}\label{prop:MCvsEU}
Let $N$ be the number of Monte Carlo simulations and $M$ the number of steps of the Euler scheme, then the tradeoff between $N$ and $M$ for computing $W_{i}(\theta)$ (and similarly $V_{i}(\theta)$) is 
\[N \asymp 
M^{\frac{\tilde{\ga}}{2\ga-1}-3},
\]
for all $1/2<\gamma<H$ and $\tilde{\gamma} = Tm(d+1)$, where $T$ is the time horizon, $m$ the dimension of the observed process and $d$ the dimension of the noise process.
\end{proposition}
\begin{proof}
We discuss the proof only for $W_{i}$, by following exactly the same steps we can obtain the same result for $V_{i}$.

We only need to check that our process $w$  satisfies the conditions of Theorem 1 in \cite{DG}.
\begin{enumerate}
\item[(i)] We can easily see that the discretized $\bar{w}^{M}_{t_{i}}$ converges uniformly to $w_{t_{i}}$.
\item[(ii)] In addition, we have bounded moments of $w_{t_{i}}$, thus 
$\mathbf{E}[\bar{W}_{t_{i}}^{2}] \rightarrow \mathbf{E}[w_{t_{i}}^{2}]$.
\item[(iii)]From Proposition \ref{thm:cvgce-Vi-Wi} we have that the rate of convergence of the Euler scheme of $\bar{w}_{t_{i}}^{M}$ is $M^{1-2\gamma}$, for $1/2<\gamma<H$. 
\item[(iv)] The computer time required to generate $\bar{w}_{t_{i}}^{M}$ is given by $\tau(1/M)$, which satisfies:
\[\tau(1/M) =  T m (d+1)M = \tilde{\gamma}M\]
where $T$ is the length of the time period, $m$ is the dimension of the SDE, $d$ is the dimension of the fBm and $M$ is the number of Euler steps.
\end{enumerate}
By applying Theorem 1 (by \cite{DG}) the optimal rule for choosing the number of Monte-Carlo simulations and the number of Euler steps is chosen such that the asymptotic error is minimized. Therefore, for $t$ the total budget of computer time, as $t$ increases, then the Euler step should converge to zero with order $\frac{1-2\gamma}{\tilde{\gamma}+2-4\gamma}$ or equivalently:
\[ \frac{1}{M} \asymp t^{\frac{1-2\gamma}{\tilde{\gamma}+2-4\gamma}} \text{ thus } t \asymp M^{-\frac{\tilde{\gamma}+2-4\gamma}{1-2\gamma}}. \]
But the number of operations needed for an arbitrary Monte Carlo simulation $t_0$ is equal to $\tilde{\gamma} M N$. Thus, we finally obtain that
$N \asymp M^{-\frac{\tilde{\gamma}+2-4\gamma}{1-2\gamma}-1}$.
\end{proof}

\subsection{Discretization of the score function}

Consider the following discretized version of the score function, i.e. $\nabla_{l}\ell_{n}(\te)$:
\begin{equation}
{\hat\nabla_{l}\ell_{n}(\te)} = \frac{\hat{V_{i}}}{\hat{W_{i}}} := 
\frac{\frac{1}{N} \sum_{k=1}^{N}{\bar v_{i,k}^{M}}}{\frac{1}{N} \sum_{k=1}^{N}{\bar w_{i,k}^{M}}},
\end{equation}
where ${\bar w_{1,k}^{M}},{\bar w_{2,k}^{M}}, \ldots$ and ${\bar v_{1,k}^{M}},{\bar v_{2,k}^{M}}, \ldots$ are iid copies of ${\bar w_{i}^{M}}$ and ${\bar v_{i}^{M}}$ respectively. Our aim in this section is to give a global bound for the mean square error obtained by approximating $\nabla_{l}\ell_{n}(\te)$ by ${\hat\nabla_{l}\ell_{n}(\te)}$.

\begin{proposition}\label{prop:scoreOrder}
The discretized score function ${\hat\nabla_{l}\ell_{n}(\te)}$ converges to the continuous score function $\nabla_{l}\ell_{n}(\te)$ with rate of convergence of order $M^{-(2\gamma-1)}$, where $1/2<\gamma<H$ and $M$ is the number of Euler steps used in the discretization. 
\end{proposition}
\begin{proof}
We discuss the idea of the proof for the $W_{i}$ term first:
\begin{eqnarray*}
\mathbf{E}\left( \hat{W_{i}} - W_{i}\right)^{2} &=& \mathbf{E}\left( \frac{1}{N} \sum_{k=1}^{N} \bar{w}^{M}_{i,k}  - \mathbf{E}[w_{i}(\theta)]\right)^{2}\\
&=& \mathbf{E}\biggl( \frac{1}{N} \sum_{k=1}^{N} \bar{w}^{M}_{i,k} -  \frac{1}{N} \sum_{k=1}^{N} w_{i,k}  +  \frac{1}{N} \sum_{k=1}^{N} w_{i,k} - \mathbf{E}[w_{i}(\theta)]\biggr)^{2}.
\end{eqnarray*}
Thanks now to the independence property between Monte Carlo runs, we get
\begin{align*}
&\mathbf{E}\left( \hat{W_{i}} - W_{i}\right)^{2} 
\leq \frac{2}{N} \sum_{k=1}^{N} \mathbf{E}(\bar{w}^{M}_{i,k} - w_{i,k})^{2} \;+\; 2\, \mathbf{E}\biggl( \frac{1}{N} \sum_{k=1}^{N} w_{i,k} - \mathbf{E}[w_{i}(\theta)]\biggr)^{2}\\
&=\; \frac{1}{N} \sum_{k=1}^{N} \text{(Euler MSE)}^{2} \;+\; \text{(Monte Carlo MSE)}^{2}
\asymp (M^{1-2\gamma})^{2} + \frac{1}{N},
\end{align*}
and thus
\[{\rm MSE} \left( \hat{W_{i}} - W_{i}\right) \asymp \sqrt{(M^{1-2\gamma})^{2} + \frac{1}{N}}.\]

Now, if we use Proposition \ref{prop:MCvsEU}, i.e. $N \asymp M^{-\frac{\tilde{\gamma} +2-4\gamma} {1-2\gamma} -1}$,
for all $1/2<\gamma<H$, and $\tilde{\gamma} = Tm(d+1)$, where $T$ is the time horizon, $m$ the dimension of the observed process and $d$ the dimension of the noise process, we have
\[MSE \left( \hat{W_{i}} - W_{i}\right) \asymp \sqrt{M^{2-4\gamma} + M^{\frac{\tilde\ga}{1-2\ga}+3}}\asymp M^{1-2\gamma}, \]
since the first is the dominant term above.

Following the same procedure, we can show that ${\rm MSE}( \hat{V_{i}} - V_{i}) \asymp M^{1-2\gamma}$ and thus the claim of the proposition follows easily.

\end{proof}

\begin{remark}
In Proposition \ref{prop:scoreOrder} the rate of convergence is independent of the dimension of the problem, i.e. it is independent of the parameter $\tilde{\gamma} = T m (d+1)$.
\end{remark}

\section{Numerical Examples}
\label{sec:num-examples}

In this section our aim is to investigate the performance of the suggested maximum likelihood method in practice. We study the one-dimensional fractional Ornstein-Uhlenbeck process, a linear two-dimensional system of fractional SDEs and then some real data given by a financial time series. Before presenting our results, we first discuss some technical issues raised by the algorithmic implementation of our method.

The goal is to find the root of the quantity $\nabla_{l}\ell_{n}(\theta)$ with respect to $\theta$. We can divide  this procedure in two parts. The first part consists in computing the root of the log-likelihood using a stochastic approximation algorithm. This is a stochastic optimization technique firstly introduced by Robbins and Monro (1951) that is used when only noisy observations of the function are available. 
In our case it is appropriate, since we want to solve
\[\nabla_{l}\ell_{n}(\theta) = 0,\]
where $\nabla_{l}\ell_{n}(\theta)$ is given by Theorem \ref{prop:loglik} and has to be approximated by ${\hat\nabla_{l}\ell_{n}}(\te)$. Thus, the recursive procedure is of the following form
\begin{equation}\label{eq:iter}
\hat{\theta}_{k+1} = \hat{\theta}_{k} - a_{k} {\hat\nabla_{l}\ell_{n}} (\hat{\theta}_{k}).
\end{equation}
where ${\hat\nabla_{l}\ell_{n}}$ is the estimate of  $\nabla_{l}\ell_{n}$ at the k-th iteration based on the observations and $a_{k}$ is a sequence of real numbers such that $\sum_{k=1}^{\infty}a_k=\infty$ and $\sum_{k=1}^{\infty}a_k^2<\infty$. Under appropriate conditions (see for example \cite{BL}), the iteration in (\ref{eq:iter}) converges to $ \theta$ almost surely. The step sizes satisfy $ a_{k} > 0$ and the way that we choose them can be found in~\cite{KY}.

The second part consists of the computation of ${\hat\nabla_{l}\ell_{n}} (\hat{\theta}_{k})$ at each step of the stochastic approximation algorithm. Thus, for a given value of $\theta_{k}$ (the one computed at the $k$-th iteration) we want to compute ${\hat\nabla_{l}\ell_{n}}(\theta_{k})$ when we are given $n$ discrete observations of the process:  $y_{t_{i}}$, $i=1,\ldots, n$. Here, we describe the main idea of the algorithm we use for only one step. Thus, assume that we are at $[t_{i-1}, t_{i}]$, and at time $t_{i}$ we obtain the $i$-th observation. We want to compute $W_{i}(\theta)$ and $V_{i}(\theta)$ according to expressions~\eqref{eq:expr-W-i} and \eqref{eq:expr-V-i} respectively.  To compute the expectations we use simple Monte-Carlo simulations.Therefore, we discretize the time interval into $N$ steps
\[t_{i-1} = s_{0} < s_{1} < \cdots < s_{N} = t_{i}.\]
From each simulated path (apart from that of fBm) we only need to keep the terminal value which is the value of the process at time $t_{i}$. The algorithm is the following

\begin{enumerate}
\item \label{step:fbmsim} Simulate $N$ values of fBm in the interval $[t_{i-1}, t_{i}]$ using for example the circulant matrix method (any exact -preferably- simulation technique can be used).
\item \label{step:Ysim} Using the simulated values from step \ref{step:fbmsim} and an Euler scheme for the SDE (1), simulate the value of the process at time $t_{i}$. For example, for $k=0,\ldots, N$
\begin{equation*}
\bar{Y}_{s_{k}}^{M} = \bar{Y}_{s_{k-1}}^{M} + \mu(\bar{Y}_{s_{k}}^{M}) (s_{k}-s_{k-1}) + \sum_{j=1}^{d} \sigma^{(j)}(\bar{Y}_{s_{k-1}}^{M}) (B^{(j)}_{s_{k}} -   B^{(j)}_{s_{k-1}}).
\end{equation*}
\item \label{step:Wind}Using step \ref{step:Ysim} and the observation at time $t_{i}$, compute the indicator function $\mathbf{1}_{(Y_{t_{i}}(\theta)> y_{t_{i}})}$.
\item \label{step:DYsim} Using  step \ref{step:fbmsim} and an Euler scheme simulate  $D_{t_{i}}Y_{\tau}^{i}$, as given in Lemma 3.3 for $n=1$ -first Malliavin derivative-.
\item \label{step:etasim} Using  step \ref{step:fbmsim} and an Euler scheme simulate $\eta_{t_{i}}^{kj}$, $k, j=1,\ldots,m$, as given in Proposition 3.7.
\item Steps \ref{step:DYsim} and \ref{step:etasim} are used to compute $Q_{st_{i}}^{pj}$, $p\in\{1,\ldots,m\}$, $j\in\{1,\ldots,d\}$ as defined in Propositions \ref{prop:def-H-j} and \ref{prop:comput-H}.
\item Simulate the Malliavin derivative of the product $D_s [Y_{t} Q_{rt}^{pj}]$.
\item \label{step:Up} Using the previous steps, numerical integration for the double integral and numerical integration for the stochastic integral we compute $U_{p}(Y_{t_{i}}(\theta))$ as defined in Proposition \ref{prop:def-H-j}.
\item \label{step:HW} Recursively compute $H_{(1,\ldots,m)}(Y_{t_{i}}(\theta))$ as given in (19).
\item \label{step:lastA} Combine steps \ref{step:Wind} and \ref{step:HW} to obtain the kernel $W_{i}(\theta)$.
\item \label{step:V_i} We repeat steps \ref{step:fbmsim} through \ref{step:lastA} $N$ times and we average these values to obtain an estimate for the expectation $W_{i}(\theta)$. 
\end{enumerate}
Using a similar procedure we can obtain an estimate for the expectation $V_{i}(\theta)$. Finally, for each $i=1, \ldots, n$ we compute $V_{i}(\theta)/W_{i}(\theta)$ and sum over $i$ to obtain the desired value of the log-likelihood at $\theta_{k}$.

\smallskip

We have completed the study of our numerical approximation of the log-likelihood, and are now ready for the analysis of some numerical examples.

\subsection{Fractional Ornstein-Uhlenbeck process}
Though our method can be applied to highly nonlinear contexts, we focus here on some linear situations, which allow easier comparisons with existing methods or exact computations. Let us first study the one-dimensional fractional Ornstein-Uhlenbeck process, i.e.
\begin{equation}\label{eq:OU}
dY_{t} = -\lambda Y_{t} dt \; + \; dB_{t},
\end{equation}
where the solution is given $Y_{t}(\lambda) = \int_{0}^{t} e^{-\lambda (t-s)} dB_{s}$ (notice the existence of an explicit solution here). In this case our methodology is quite simplified. The log-likelihood can be written as follows:
\begin{equation*}
\partial_{\lambda}\ell(\lambda;y) = \sum_{i=1}^{n} \frac{\mathbf{E} \biggl[\partial_{\lambda} Y_{t}(\lambda) \; \mathbf{1}_{(Y_{t}(\lambda)>y)} \;H_{(1,1)}(\lambda) + \Bigl( Y_{t}(\lambda) - y \Bigr)_{+} \partial_{\lambda} H_{(1,1)}(\lambda) \biggr]}{\mathbf{E} \biggl[ \mathbf{1}_{(Y_{t}(\lambda)>y)}\; H_{(1)}\Bigl( Y_{t}(\lambda), 1\Bigr) \biggr]}.
\end{equation*}
The Malliavin derivative of $Y_{t}(\lambda)$  satisfies the following ODE
\[D_{s}Y_{t}(\lambda) = 1 - \lambda \int_{s}^{t} D_{s}Y_{u}(\lambda) du,\]
with solution
$D_{s}Y_{t}(\lambda) = e^{-\lambda\; t}\; \mathbf{1}_{\{s\leq t\}}$. The corresponding norm is
\[\| D_{\cdot} Y_{t} (\lambda) \|^{2} = c_H\; \int_{s}^{t} \int_{s}^{t} e^{-\lambda (u+v)} |u-v|^{2H-2} dudv. \]
The higher order derivatives of $Y_{t}(\lambda)$ are equal to zero. Therefore,
\begin{equation*}
H_{(1)}\Bigl( Y_{t}(\lambda) \Bigr) = \frac{1}{\left\| D_{\cdot} Y_{t}(\lambda) \right\|^{2} }\; \int_{s}^{t} e^{-\lambda u} dB_{u}
\end{equation*}
and thus
\begin{eqnarray*}
H_{(1,1)}(\lambda) &=& \frac{1}{\left\| D_{\cdot} Y_{t}(\lambda) \right\|^{4} } \int_{s}^{t}  \int_{s}^{t} e^{-\lambda (u+v)} dB_{u} dB_{v} - c_{H} \left\| D_{\cdot} Y_{t}(\lambda) \right\|^{-2}.
\end{eqnarray*}
The derivative with respect to the unknown parameter $\lambda$ satisfies
\begin{equation*}
 \partial_{\lambda} Y_{t}(\lambda) =  -\int_{0}^{t} Y_{s}(\lambda) -\la \partial_{\lambda} Y_{s}(\lambda) ds
\end{equation*}
with solution $\partial_{\lambda} Y_{t}(\lambda) = \int_{0}^{t} (t-s) e^{-\la(t-s)} dB_s$. The last term we need to compute is:
\begin{eqnarray*}
\partial_{\lambda} H_{(1,1)}(\lambda) &=& \frac{1}{\|D_{\cdot} Y_{t}(\lambda) \|^{8}} 
\biggl[ \|D_{\cdot} Y_{t}(\lambda) \|^{4} \int_{s}^{t} \int_{r}^{t} -(u+v) e^{-\lambda (u+v)} dB_{u} dB_{v} \\
&& -  2c_H\|D_{\cdot} Y_{t}(\lambda) \|^{2} \int_{s}^{t} \int_{r}^{t} -(u+v) e^{-\lambda (u+v)} |u-v|^{2H-2} du dv \biggr]\\
&& - \frac{c_H^{2} \int_{s}^{t} \int_{r}^{t} -(u+v) e^{-\lambda (u+v)} |u-v|^{2H-2} du dv  }{\|D_{\cdot} Y_{t}(\lambda) \|^{4}}.
\end{eqnarray*}

Now, we compute the MLE following the algorithm we described above. The results we obtained are summarized in the following table:\\
\begin{center}
\begin{tabular}{|c|| c| c|}
\hline 
  True $\lambda$ & MLE $\hat{\lambda}$ & Standard Error \\ \hline \hline
  0.5 & 0.497 & 0.00369\\ 
  4 & 3.861 & 0.00127\\\hline
\end{tabular}
\end{center}
\begin{remark}
The value of $H$ used for the simulation of the process is 0.6. The number of observations is $n=50$, the number of Euler steps is $M=500$, the number of stochastic approximation steps is $K=50$ and the number of MC simulations $N=500$.
\end{remark}

\subsection{Two-dimensional fractional SDE}
In this section we study the following system of fractional OU processes:
\begin{eqnarray}
dY_{t}^{(1)} = -\alpha Y_{t}^{(2)} dt \; + \;\beta dB_{t}^{(1)}\notag\\
dY_{t}^{(2)} = -\beta Y_{t}^{(1)} dt \; + \; \beta dB_{t}^{(2)}.
\end{eqnarray}
In this case, the computations are more involved even though the SDEs are linear functions of $Y$. Furthermore, the parameter we want to estimate is two-dimensional as well ($\theta = (\alpha, \beta)^T$), which complicated  the optimization procedure. Therefore, instead of computing only one  derivative, we need to compute both derivatives with respect to $\alpha$ and $\beta$ and then compute the solution of the system of two equations
\begin{equation*}
\nabla_{\alpha}\ell(\alpha, \beta;y) = 0, \qquad
\nabla_{\beta}\ell(\alpha, \beta;y) = 0,
\end{equation*}
where
\begin{multline*}
\nabla_{l}\ell(\alpha, \beta;y) = \sum_{i=1}^{n} [\mathbf{E} [ \mathbf{1}_{(Y_{t}(\alpha, \beta)>y)}\; H_{(1,2)}( Y_{t}(\alpha, \beta)) ]^{-1}\\
 \times \{\mathbf{E} [\nabla_{l} Y_{t}(\alpha, \beta) \; \mathbf{1}_{(Y_{t}(\alpha, \beta)>y)} \;H_{(1,2,1,2)}(\alpha, \beta) + ( Y_{t}(\alpha, \beta) - y )_{+} \nabla_{l} H_{(1,2,1,2)}(\alpha, \beta) ]\}
\end{multline*}
and $l=\alpha \text{ or } \beta$. The Malliavin derivative of $Y_{t}$ computes as follows:
\begin{equation*}
D_{s}Y_{t}^{(1)} = \beta - \alpha \int_{s}^{t} D_{s}Y_{u}^{(2)} du \qquad
D_{s}Y_{t}^{(2)} = \beta - \beta \int_{s}^{t} D_{s}Y_{u}^{(1)} du.
\end{equation*}
The covariance matrix $\gamma_{t}$ is given by $(\langle D_{\cdot}Y_{t}^{i}, D_{\cdot}Y_{t}^{j} \rangle)_{1\leq i,j\leq 2}$. The inverse of the covariance matrix satisfies the following SDE
\[ \gamma_{t}^{-1} = -\int_{0}^{t} [\gamma_{u}^{-1} M + M^{T} \gamma_{u}^{-1} ] du,  \]
where 
\[ M=
\begin{bmatrix}
0 & \alpha  \\
\beta & 0 
\end{bmatrix}
\]
Now, it remains to compute the quantities $H_{(1,2)}$ and $H_{(1,2,1,2)}$. This can be done using the recursive formulas in Proposition 3.12, but we need to keep in mind that higher order derivatives of $Y$ are equal to zero, thus they will be simplified. Indeed,
\[H_{(1)}(Y_{t}) = \sum_{j=1}^{2} Y_{t}  \int_{0}^{t} (\gamma_{s}^{-1})^{1j} D_{s}Y_{t}^{j} dB_{s}^{j} - c_{H} \int_{0}^{t}\int_{0}^{t} D_{s}Y_{t}^{j}Q_{rt} |r-s|^{2H-2}dr ds.\]
Moreover, we can easily see that
\[ H_{(1,2)} (Y_{t}) = H_{(1)}(Y_{t})\int_{0}^{t} Q_{st}dB_{s} - c_{H} \int_{0}^{t}\int_{0}^{t} D_{s}H_{(1)}(Y_{t}) Q_{rt}|r-s|^{2H-2}dr ds \]
\[ H_{(1,2,1,2)} (Y_{t}) = H_{(1,2,1)}(Y_{t})\int_{0}^{t} Q_{st}dB_{s} - c_{H} \int_{0}^{t}\int_{0}^{t} D_{s}H_{(1,2,1)}(Y_{t}) Q_{rt}|r-s|^{2H-2}dr ds \]
Of course, recall that $Q_{st}^{pj} = (\gamma_{s}^{-1})^{pj} D_{s}Y_{t}^{j}$. In practice, these quantities are computed recursively. The last step is to compute the derivative of $H_{(1,2,1,2)} (Y_{t})$ with respect to $\alpha$ and $\beta$, which in this case is not as complicated and compute the MLEs using the algorithm discussed in the previous section. The table below summarizes our results, and we  have plotted the corresponding histograms in Figure \ref{fig:hist}.\\
 \begin{center}
\begin{tabular}{|c |c || c |c|}
\hline 
  Parameter & True Value & MLE  & Standard Error \\ \hline \hline
  $\alpha$ & 2 & 2.003 &  0.0518\\
  $\beta$ & 4 & 3.987 &  0.0157\\ \hline
\end{tabular}

\end{center}

 \begin{figure}
\centering
\begin{tabular}
[c]{cc}
\includegraphics [width=0.45\linewidth, height=0.45\linewidth]{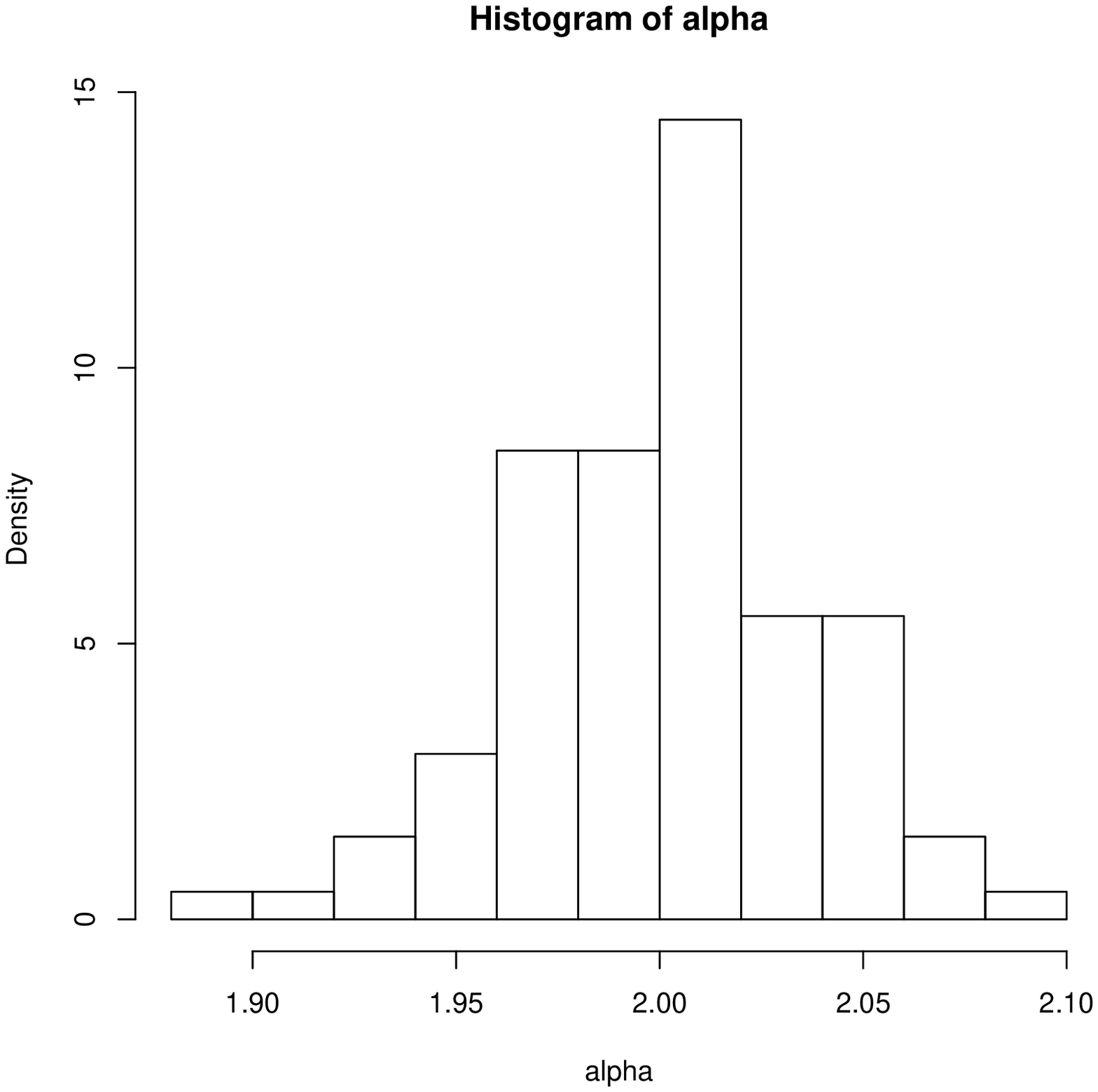} &
\includegraphics [width=0.45\linewidth, height=0.45\linewidth]{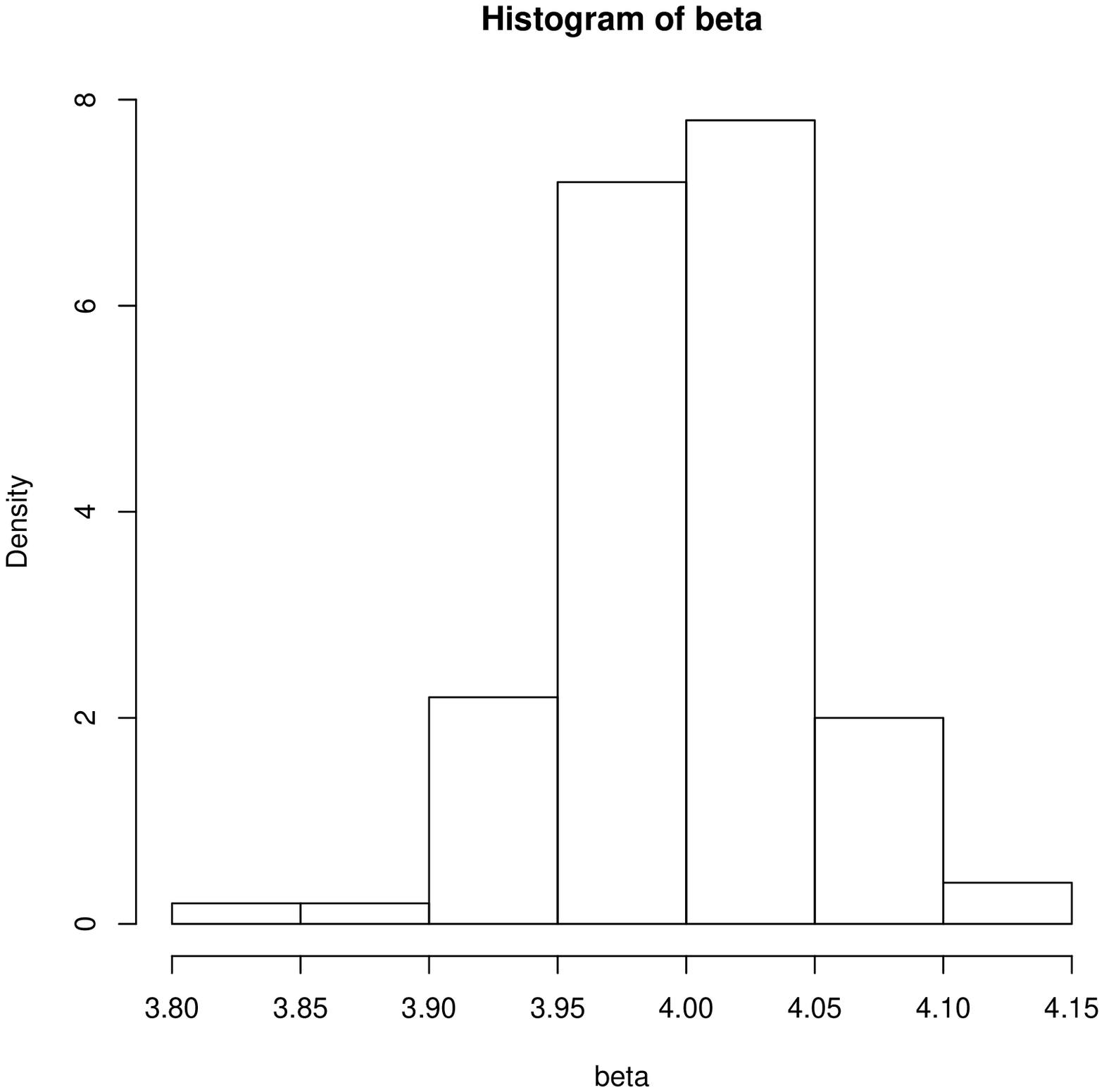}\\
Drift Parameter. & Diffusion Parameter.\\ \vspace{1cm}
\end{tabular}
\caption{Empirical Distribution of the estimators for $\alpha$ and $\beta$.}
\label{fig:hist}
\end{figure}

\begin{remark}
The value of $H$ used for the simulation of the process is 0.6. The number of observations is $n=50$, the number of Euler steps is $N=500$, the number of stochastic approximation steps is $K=50$ and the number of MC simulations $M=500$.
\end{remark}

\subsection{Application to financial data}

One of the most popular applications of fractional SDEs is in finance. Hu and Oksendal, \cite{HOk}, introduced the  fractional Black-Scholes model in order to account for inconsistencies of the existing models in practice. More specifically,  the stock price is described therein by a fractional geometric Brownian motion with Hurst parameter $1/2< H <1$. The choice of this model is based on empirical studies that displayed the presence of long-range dependence on  stock prices, for example in Willinger, Taqqu and Teverovsky, \cite{WTT}. 

However, the presence of fractional Brownian motion in the model allows for arbitrage in the general setting. It has been shown that arbitrage opportunities can be avoided in a number of ways, for example the reader can refer to Rogers \cite{Ro}, Dasgupta and Kallianpur \cite{DK} and Cheridito \cite{Ch}. We choose to model the stock price as as follows:
\begin{equation}
dS_{t} = \mu S_{t} dt  + \sigma dB_{t},
\end{equation}
where $B$ is a fractional Brownian motion with Hurst index $1/2<H<1$. For this SDE (as well as for a more general class of fractional SDEs) Guasoni, \cite{Gua}, proved that  there is no arbitrage when transaction costs are present. 

 Our goal is to estimate  the unknown parameters $\mu$ and $\sigma$ based on daily observations of the S\&P 500 index (data from June 2010 until December 2010). Since the Hurst parameter is piece-wise constant, we devide the data in three groups (of 50 daily observations each) and we compute for each one the Hurst index using the Rescaled-Range (R/S) statistic. We obtain that for the first group of data $\hat{H_{1}}=0.59$, for the second $\hat{H_{2}}=0.63$ and for the third one $\hat{H_{3}}=0.61$. For the different groups, we apply our maximum likelihood approach in order to estimate $\mu$ and $\sigma$. The estimates are summarized in the following table:
\begin{center}
\begin{tabular}{|c || c| c| c|}
\hline 
 Estim. Parameters & Group 1: $\hat{H_{1}}=0.59$  & Group 2: $\hat{H_{2}}=0.63$  & Group 3: $\hat{H_{3}}=0.61$ \\ \hline \hline
  $\hat{\mu}$ & 0.015 (0.0123) & 0.019 (0.0144) & 0.011 (0.0214) \\ 
  $\hat{\sigma}$ & 0.352 (0.058) & 0.339 (0.046) & 0.341 (0.024) \\ \hline
\end{tabular}
\end{center}
\begin{remark}
The volatility is computed in years. In addition, during this period of time the historical volatility is around 0.38, which is coherent with our own estimation.
\end{remark}

\end{document}